\documentclass[12pt,xcolor=svgnames]{amsart}

\usepackage{amsmath, amsthm, amssymb, amsfonts, verbatim}

\usepackage{xspace,mathtools,bm}

\usepackage[svgnames]{xcolor}

\usepackage[pagebackref=true, colorlinks=true, citecolor=blue]{hyperref}

\def\0{{\bf 0}}

\def\R{{\mathbb R}}
\def\Z{{\mathbb Z}}

\def\N{{\mathbb N}}

\theoremstyle{plain}

\newtheorem{theorem}{Theorem}[section]

\newtheorem{prop}[theorem]{Proposition}
\newtheorem{lem}[theorem]{Lemma}
\newtheorem{lemma}[theorem]{Lemma}

\theoremstyle{definition}
\newtheorem{definition}[theorem]{Definition}

\newtheorem*{theorem*}{Theorem}

\theoremstyle{remark}
\newtheorem{remark}[theorem]{Remark}

\newcommand{\Jim}[1]{{\color{DarkRed}#1}}

\newcommand{\abs}[1]{\left\vert#1\right\vert}

\DeclareMathOperator{\dv}{div}
\DeclareMathOperator{\curl}{curl}
\DeclareMathOperator{\supp}{supp}
\newcommand{\iny}{\ensuremath{\infty}}
\newcommand{\grad}{\ensuremath{\nabla}}
\newcommand{\prt}{\ensuremath{\partial}}
\newcommand{\brac}[1]{\ensuremath{\left[ #1 \right]}}
\newcommand{\pr}[1]{\ensuremath{\left( #1 \right) }}

\newcommand{\norm}[1]{\ensuremath{\left\Vert #1 \right\Vert}}

\newcommand{\smallnorm}[1]{\ensuremath{\Vert #1 \Vert}}
\newcommand{\Cal}[1]{\ensuremath{\mathcal{#1}}}

\newcommand{\diff}[2]{\frac{ d#1}{d#2}}
\newcommand{\al}{\alpha}

\newcommand{\stardot}{\Jim{\mathop{* \cdot}}}

\DeclareMathOperator{\PV}{p.v.} %

\allowdisplaybreaks

\begin{document}

\raggedbottom

\numberwithin{equation}{section}

%
%
\newcommand{\MarginNote}[1]{
    \marginpar{
        \begin{flushleft}
            \footnotesize #1
        \end{flushleft}
        }
    }
%
%
\newcommand{\NoteToSelf}[1]{
    }

%
%
\newcommand{\Obsolete}[1]{
    }

\newcommand{\Detail}[1]{
    \MarginNote{Detail}
    \skipline
    \hspace{+0.25in}\fbox{\parbox{4.25in}{\small #1}}
    \skipline
    }

\newcommand{\Comment}[1] {
    \skipline
    \hspace{+0.25in}\fbox{\parbox{4.25in}{\small \textbf{Comment}: #1}}
    \skipline
    }

%
%

\newcommand{\IntTR}
    {\int_{t_0}^{t_1} \int_{\R^d}}

\newcommand{\IntAll}
    {\int_{-\iny}^\iny}

\newcommand{\Schwartz}
    {\ensuremath \Cal{S}}

\newcommand{\SchwartzR}
    {\ensuremath \Schwartz (\R)}

\newcommand{\SchwartzRd}
    {\ensuremath \Schwartz (\R^d)}

\newcommand{\SchwartzDual}
    {\ensuremath \Cal{S}'}

\newcommand{\SchwartzRDual}
    {\ensuremath \Schwartz' (\R)}

\newcommand{\SchwartzRdDual}
    {\ensuremath \Schwartz' (\R^d)}

\newcommand{\HSNorm}[1]
    {\norm{#1}_{H^s(\R^2)}}

\newcommand{\HSNormA}[2]
    {\norm{#1}_{H^{#2}(\R^2)}}

\newcommand{\Holder}
    {H\"{o}lder }

\newcommand{\Holders}
    {H\"{o}lder's }
    
\newcommand{\HolderAllCaps}
    {H\"{O}LDER }
    
\newcommand{\HolderZygmund}{H\"{o}lder-Zygmund }
    
\newcommand{\Poincare}
	{Poincar\'{e}\xspace}

\newcommand{\Poincares}
	{Poincar\'{e}'s\xspace}
	
\newcommand{\Calderon}
    {Calder\'{o}n\xspace}

\newcommand{\Gronwalls}
    {Gr\"{o}nwall's\xspace}

\newcommand{\sg}
    {\ensuremath{\sigma}}    
    
    \newcommand{\be}
    {\ensuremath{\beta}}

   \newcommand{\pib}
    {\ensuremath{\phi_{\beta}}} 
    
     \newcommand{\sib}
    {\ensuremath{\psi_{\beta}}} 
    
        \newcommand{\besw}
    {\ensuremath{ {\dot{B}}^0_{2,2}(w) }} 
    
           \newcommand{\beswhalf}
    {\ensuremath{ {\dot{B}}^{1 \slash 2}_{2,2}(w) }} 
    
         \newcommand{\besdwhalf}
    {\ensuremath{ {\dot{B}}^{1 \slash 2}_{2,2}(Dw) }} 
       
           \newcommand{\besdw}
    {\ensuremath{ {\dot{B}}^0_{2,2}(Dw) }}

\title
    [fluid equations in \HolderAllCaps and Uniformly Local Sobolev Spaces]
    {Existence of Solutions to fluid equations in \HolderAllCaps and Uniformly Local Sobolev Spaces}

\author{David M. Ambrose}
\address{Department of Mathematics, Drexel University}
\curraddr{}
\email{dma68@drexel.edu}

\author{Elaine Cozzi}
\address{Department of Mathematics, Oregon State University}
\curraddr{}
\email{cozzie@math.oregonstate.edu}

\author{Daniel Erickson}
\address{Department of Mathematics, Oregon State University}
\curraddr{}
\email{ericdani@math.oregonstate.edu}

\author{James P. Kelliher}
\address{Department of Mathematics, University of California, Riverside}
\curraddr{}
\email{kelliher@math.ucr.edu}

\subjclass{Primary 76D05, 76C99} 
\date{} 


\keywords{Fluid mechanics, Euler equations}

\begin{abstract}
We establish short-time existence of solutions to the surface quasi-geostrophic equation in both the \Holder spaces $C^r(\R^2)$ for $r>1$ and the uniformly local Sobolev spaces $H^s_{ul}(\R^2)$ for $s\geq 3$.  Using methods similar to those for the surface quasi-geostrophic equation, we also obtain short-time existence for the three-dimensional Euler equations in uniformly local Sobolev spaces.  
\end{abstract}

\maketitle
Last updated: \today

\tableofcontents


\section{Introduction}

\subsection{Background}
We study non-decaying solutions of two fundamental models of fluid motion, the two-dimensional surface quasi-geostrophic equation ($SQG$) and the three-dimensional incompressible Euler equations $(E)$. Classically, these equations (without forcing) can be written
\begin{align*}
	(SQG)
		\qquad
	\begin{cases}
		\prt_t \theta + u \cdot \grad \theta  = 0
			&\text{in } [0, T] \times \R^2, \\
		u = \grad^\perp (-\Delta)^{-\frac{1}{2}} \theta
			&\text{in } [0, T] \times \R^2, \\
		\theta|_{t = 0} = \theta^0
			&\text{in } \R^2
	\end{cases}
\end{align*}
and, in velocity formulation,
\begin{align*}
	(E)
		\qquad
	\begin{cases}
		\prt_t u + u \cdot \grad u + \grad p = 0
			&\text{in } [0, T] \times \R^3, \\
		\dv u = 0
			&\text{in } [0, T] \times \R^3, \\
		u|_{t = 0} = u^0
			&\text{in } \R^3.
	\end{cases}
\end{align*}

In $(SQG)$, the scalar field $\theta$ is transported by the velocity field $u$, with $u$ recovered from $\theta$ via the constitutive law $u = \grad^\perp (-\Delta)^{-\frac{1}{2}} \theta$ (making $u$ divergence-free). In $(E)$, the velocity field $u$ is, in effect, transported by itself under the constraint that it remain divergence-free, which introduces the pressure gradient.

The parallels between these two equations become clearer when $(E)$ is written in vorticity form:
\begin{align*}
	(E_\omega)
		\qquad
	\begin{cases}
		\prt_t \omega + u \cdot \grad \omega = \omega \cdot \grad u
			&\text{in } [0, T] \times \R^3, \\
		u = K * \omega,
			&\text{in } [0, T] \times \R^3, \\
		\omega|_{t = 0} = \omega^0
			&\text{in } \R^3.
	\end{cases}
\end{align*}
Here, $\omega = \curl u$ is the vorticity, $K$ is the Biot-Savart kernel, and $u = K * \omega$ is the constitutive law. Rather than just being transported as $\theta$ is in $(SQG)$, the vorticity field is 
stretched as it is being transported. Moreover, though both constitutive laws, $(E_\omega)_2$, $(SQG)_2$, yield divergence-free vector fields, they differ sharply in that $u$ gains one more spatial derivative of regularity over that of $\omega$ for $(E_\omega)$, while it has the same spatial regularity as $\theta$ for $(SQG)$.

Each of $(SQG)$ and $(E)$ are well-posed when the data is sufficiently smooth and sufficiently decaying. Insufficient smoothness motivates various weak formulations of the equations, a long tradition in PDE. Such weak formulations leave the constitutive law alone or integrate it into the weak formulation, but generalize or weaken what it means for the PDE itself to hold (that is, $(SQG)_1$, $(E)_1$, or $(E_\omega)_1$). Studying PDEs when the data lacks sufficient decay has a shorter history, but focuses on extending or weakening the constitutive law. (Of course, both can be done at the same time.)

In this work, we study $(SQG)$ and $(E)$ for non-decaying, but sufficiently smooth solutions, which requires us to adapt the constitutive law while leaving the PDE itself unchanged. We will work with $(E)$ primarily in vorticity form, though will also use the velocity form, which requires us to obtain estimates on the pressure $p$. The constitutive law $u = K * \omega$ will enter (in adapted form) in the process of closing our estimates, as we shall see.

Our methodology for adapting the constitutive law follows that first employed by Serfati in \cite{Serfati1} for the 2D Euler equations. He obtained an identity by applying a cutoff function to the Biot-Savart kernel $K$ to separate the near-field and far-field effects of the convolution. The far-field term is then integrated by parts twice---when the PDE and constitutive law permit this, as they do for $(SQG)$ in 2D as well as $(E)$ in any dimension---which allows the integrated form of it to be controlled for non-decaying data. The resulting identity then forms, in effect, a replacement constitutive law. This can be seen clearly in the form of these identities in Lemmas \ref{L:SerfatiSQG} and \ref{L:SertatiID3DEuler}.

Even for decaying data, obtaining the existence of weak solutions to 3D Euler is beyond current technology, so we work with solutions having sufficient smoothness. We work, then, in H\"{o}lder-Zygmund spaces, which differ from \Holder spaces for integer indices---see Section \ref{S:Holder}---and in uniformly local Sobolev spaces $H^s_{ul}$ (see Section \ref{S:Hsul}). 

We prove existence for both ($SQG$) and ($E$) in $H^s_{ul}$ by applying the existence theory in H\"{o}lder-Zygmund spaces to construct an approximation sequence, developing bounds uniform with respect to the approximation parameter, and passing to the limit.

\subsection{Main results}
We state our main results in Theorems \ref{T:Main1} and \ref{T:Main2}, more completely stated in Theorems \ref{SQGHolder}, \ref{SQGUnifLocalSobolev}, and \ref{EulerUnifLocalSobolev}. See Sections \ref{S:Holder} and \ref{S:Hsul} for the definitions of the function spaces $C^r$, $\dot{C}^r$, and $H^s_{ul}$.

\begin{theorem}\label{T:Main1}
	Let $\theta^0 \in C^{r}(\R^2)$, $r \in (1, \infty)$,
	and let $u^0 \in C^r(\R^2)$ satisfy
$
u^0 =  \nabla^{\perp} (-\Delta)^{-1/2}  {\theta}^0 \text{  in }\dot{C}^r(\R^2).
$
There exists $T>0$ and a unique solution $(u, \theta)$ to $(SQG)$ with the constitutive law in the form
\begin{align*}
&u(t) = u^0
+(a\Phi) * \nabla^{\perp}(\theta(t) - \theta^0)
- \int_0^t (\nabla\nabla^\perp((1-a)\Phi)) \stardot (\theta u)
\end{align*}  
satisfying, for any $r' \in (0, r)$,
\begin{equation*}
\begin{split}
&\theta\in L^{\infty}(0,T; C^{r}(\R^2)) \cap Lip([0,T] ; C^{r-1}(\R^2)) \cap C([0,T]; C^{r'}(\R^2)),\\
&u \in L^{\infty}(0,T; C^{r}(\R^2)) \cap C([0,T]; C^{r'}(\R^2)).  
\end{split}
\end{equation*} 

If $\theta^0 \in H^s_{ul}(\R^2)$ and $u^0 \in H^s_{ul}(\R^2)$ for some $s \ge 3$ satisfy
$u^0 =  \nabla^{\perp} (-\Delta)^{-1/2} {\theta}^0$ in $\dot{C}^{\alpha}(\R^2)$,
where $\alpha > 1$ satisfies the embedding $H^s_{ul}(\R^2)\hookrightarrow {C}^{\alpha}(\R^2)$, then
\begin{equation*}
\begin{split}
&\theta \in L^{\infty}(0,T; H^s_{ul}(\R^2)) \cap Lip([0,T] ; H^{s-1}_{ul}(\R^2)),\\
&u \in L^{\infty}(0,T; H^s_{ul}(\R^2)). 
\end{split}
\end{equation*}    
\end{theorem}

\begin{theorem}\label{T:Main2}
Let $u^0 \in H^s_{ul}(\R^3)$ for some $s \ge 3$, and let $\omega^0=\nabla\times u^0$.  There exists $T>0$ and a unique classical solution $(u, p)$ to (E) satisfying 
\begin{equation*}
\begin{split}
&u \in L^{\infty}(0,T; H^s_{ul}(\R^3)) \cap Lip([0,T] ; H^{s-1}_{ul}(\R^3)).
\end{split}
\end{equation*}    
\end{theorem}

\subsection{Prior work}
There are a number of approaches to studying non-decaying solutions of nonlinear systems of partial differential equations, one of which is to focus on rough solutions, while another is to study more regular solutions. 

For rough data, there is prior work on non-decaying solutions of the two-dimensional
Euler equations under the assumption that the initial velocity and initial vorticity
are only in $L^{\infty}.$  This approach was pioneered by Serfati \cite{Serfati1}, and
extended to contexts such as exterior domains by two of the authors and collaborators
\cite{AKLN}.

Wu has previously developed existence theory for ($SQG$) in H\"{o}lder spaces \cite{Wu}, with the restriction that the initial data is not only in a H\"{o}lder space but also in an $L^{q}$ space for some $q<\infty.$  In the present work, by incorporating estimates which stem from the
Serfati identity, we remove this assumption that the data are in $L^{q},$ finding existence of non-decaying H\"{o}lder solutions for ($SQG$).

In recent work, C\'{o}rdoba and Mart\'{i}nez-Zoro \cite{CM} have shown non-existence of solutions for ($SQG$) with data in H\"{o}lder spaces $C^k$ for integer $k \ge 2$.  This is not a contradiction to the present work, for although H\"{o}lder-Zygmund spaces coincide with \Holder spaces for non-integer exponents, they are larger than \Holder spaces for integer indices. This is discussed in more detail in Section \ref{preliminary} below.  The same situation, non-existence of solutions in classical H\"{o}lder spaces but existence instead in H\"{o}lder-Zygmund spaces, has been shown
to hold for the incompressible Euler equations as well \cite{Chemin2, Chemin1}.

Majda sketches a proof of existence for the compressible Euler equations in uniformly local Sobolev spaces in \cite{Majda}; 
Majda remarks that the approach of \cite{Majda} does not work for the incompressible case.
Other work for existence of fluid equations in the uniformly local Sobolev spaces 
includes a series of papers by Zelik, Anthony and Zelik, and Chepyzhov and Zelik
on the Navier-Stokes equations, the damped Euler equations, and the damped Navier-Stokes equations, all in two spatial dimensions \cite{anthonyZelik}, \cite{chepyzhovZelik}, \cite{Zelik2007}, \cite{Zelik2013}.
Alazard, Burq, and Zuily have proved well-posedness of the gravity water waves system (i.e. the incompressible, irrotational Euler
equations with the fluid region bounded above by a free surface, subject to gravity) in uniformly local Sobolev spaces
\cite{alazard}; of course the water waves system is dispersive, and is thus of a different character than the systems studied
in the present work.  Uniformly local solutions of the water waves system were then further studied by Nguyen \cite{nguyen}.

\subsection{Organization of the paper} We define \Holder-Zygmund spaces and uniformly local Sobolev spaces in Section \ref{preliminary}, and introduce notation and provide some key lemmas. In Section \ref{S:SQGHolder}, we obtain existence of solutions to $(SQG)$ in \Holder spaces, and then employ this result in Section \ref{S:SQGSobolev} to construct an approximation sequence to obtain existence to $(SQG)$ in uniformly local Sobolev spaces. In Section \ref{S:Euler3D} we obtain existence of solutions to the 3D Euler equations in uniformly local Sobolev spaces.

In the appendices, we establish Serfati-like identities for $(SQG)$ and 3D Euler, a constitutive relation for $(SQG)$, and a pressure identity for 3D Euler akin to one used in 2D in \cite{Serfati}.

%
%
\section{Definitions and preliminary lemmas}\label{preliminary}
In this section, we state some notation, definitions, and lemmas that will be useful in what follows.  

We let $a:\R^d\rightarrow \R$, $d \ge 2$, denote a radially symmetric, smooth, compactly supported cutoff function which is identically $1$ in a neighborhood of the origin and which vanishes outside of the ball of radius $2$.   For each $\lambda>0$ and each $x\in\R^d$, we let  $a_{\lambda} (x) = a(x/\lambda)$.    

Define $G$ on $\R^3$ by
\begin{align}\label{e:G}
	G(x)
		&=  \frac{1}{4 \pi \abs{x}},
\end{align}
the fundamental solution to the Laplacian in $\R^3$, meaning that $\Delta G = \delta$, the Dirac delta function. We use $\Phi$ to denote the fundamental solution of the fractional Laplacian $(-\Delta)^{1/2}$ on $\R^2$; that is,    
\begin{equation*}
	\Phi(x) = \frac{C}{|x|}
\end{equation*}
for a constant $C>0$.  Finally, we have the simple estimates,
\begin{align}\label{e:GPhiBounds}
	\begin{split}
	\norm{a_\lambda \Phi}_{L^1(\R^2)}
		\le \lambda, \quad
	\norm{\grad\grad^{\perp} ((1 - a_\lambda) \Phi)}_{L^1(\R^2)}
		\le C \lambda^{-1}.
	\end{split}
\end{align}

\subsection{The Littlewood-Paley operators}\label{S:LPOperators}
In Section \ref{S:SQGHolder}, we establish existence of solutions to ($SQG$) in the spaces $C^r(\R^2)$ for $r>1$, where $C^r(\R^2)$ is defined using the Littlewood-Paley decomposition.  We therefore begin this section with an overview of the Littlewood-Paley operators and some of their properties.
It is classical that there exists two functions ${\chi}, {\varphi} \in \mathcal{S}(\R^d)$ with supp $\hat{\chi}\subset \{\xi\in \R^d: |\xi |\leq \frac{5}{6} \}$ and supp $\hat{\varphi}\subset \{\xi\in \R^d: \frac{3}{5} \leq|\xi |\leq \frac{5}{3} \}$, such that, if for every $j\in\Z$ we set $\varphi_j(x) = 2^{jd} \varphi(2^j x)$, then
\begin{equation*}
\begin{split}
	&\hat{\chi}+ \sum_{j\geq 0} \hat{\varphi_j}
		= \hat{\chi} + \sum_{j\geq 0} \hat{\varphi}(2^{-j} \cdot) 
		\equiv 1.
\end{split}
\end{equation*}

For $n\in\Z$, define ${\chi}_n \in \mathcal{S}(\R^d)$ in terms of its Fourier transform ${\hat{\chi}}_n$, where ${\hat{\chi}}_n$ satisfies 
\begin{equation*}
{\hat{\chi}}_n (\xi) =   \hat{\chi}(\xi) + \sum_{j\leq n} \hat{\varphi}_j(\xi)
\end{equation*}
for all $\xi\in\R^d$.  For $f\in \mathcal{S}'(\R^d)$, define the operator $S_n$ by  
\begin{equation*}
S_n f = {{\chi}}_n \ast f.
\end{equation*}
Finally, for $f\in \mathcal{S}'(\R^d)$ and $j\in\Z$, define the inhomogeneous Littlewood-Paley operators ${\Delta}_j$ by
\begin{align*}
    \begin{matrix}
        &\Delta_j f  = \left\{
            \begin{matrix}
                     0, \qquad j<-1\\
                 \chi\ast f,  \qquad j=-1\\
                \varphi_j\ast f, \qquad j\geq 0,
            \end{matrix}
            \right.
    \end{matrix}
\end{align*}
and, for all $j\in\Z$, define the homogeneous Littlewood-Paley operators $\dot{\Delta}_j$ by
\begin{equation*}
    \dot{\Delta}_j f = {\varphi}_j \ast f.
\end{equation*}  
Note that $\dot{\Delta}_j f = {\Delta}_j f$ when $j\geq 0$.

We will make use of Bernstein's Lemma in what follows.  A proof of the lemma can be found in \cite{Chemin1}, Chapter 2.  Below, $C_{a,b}(0)$ denotes the annulus with inner radius $a$ and outer radius $b$.  
\begin{lem}\label{bernstein}
(Bernstein's Lemma) Let $r_1$ and $r_2$ satisfy $0<r_1<r_2<\infty$, and let $p$ and $q$ satisfy $1\leq p \leq q \leq \infty$. There exists a positive constant $C$ such that for every integer $k$, if $u$ belongs to $L^p(\R^d)$, and supp $\hat{u}\subset B_{r_1\lambda}(0)$, then 
\begin{equation}\label{bern1}
\sup_{|\alpha|=k} ||\partial^{\alpha}u||_{L^q} \leq C^k{\lambda}^{k+d(\frac{1}{p}-\frac{1}{q})}||u||_{L^p}.
\end{equation}
Furthermore, if supp $\hat{u}\subset C_{r_1\lambda, r_2\lambda}(0)$, then 
\begin{equation}\label{bern2}
C^{-k}{\lambda}^k||u||_{L^p} \leq \sup_{|\alpha|=k}||\partial^{\alpha}u||_{L^p} \leq C^{k}{\lambda}^k||u||_{L^p}.
\end{equation} 
\end{lem}

\begin{lem}\label{L:CZLike}
	Let $\Psi(x) = C \abs{x}^{1 - d}$ on $\R^d$.
	There exists $C > 0$ such that for every $j \in \Z$,
	\begin{align}\label{withouta}
		\smallnorm{\dot{\Delta}_j  (\grad \Psi* f)}_{L^\iny(\R^d)}
			\le C \smallnorm{\dot{\Delta}_j f}_{L^\iny(\R^d)}.
	\end{align}
	The result holds with $\grad \Psi$ replaced by $\grad (a\Psi)$. 
\end{lem}
\begin{proof}
The proof of (\ref{withouta}) follows from an argument identical to the proof of Lemma 8 in \cite{CK}.  To see that the result holds for $\grad (a\Psi)$ in place of $\grad \Psi$, first note that the equivalent of this lemma for a \Calderon-Zygmund operator $T$ is well-known \cite{Stein}.
	We note, however, that $T = \grad (a \Psi) *$ is not quite a \Calderon-Zygmund operator;
	rather (see, for instance, Proposition 6.1 of \cite{BK}),
	\begin{align*}
		\grad (a \Psi) * f(x)
			&= \PV \int_{\R^d} \grad (a \Psi) (x - y) f(y) \, dy
				+ C f(x) I,
	\end{align*}
	where the principal value integral does represent a \Calderon-Zygmund operator.
	The result then follows immediately.
\end{proof}

\begin{remark}\label{R:Convolutions}
	The convolution $\grad (a \Psi) * f$ in Lemma \ref{L:CZLike} is that of a compactly supported
	distribution with a distribution. As in Theorem 6.37(e) of \cite{R1991},
	we can move derivatives on and off each factor, so
	\begin{align*}
		\grad (a \Psi) * f
			= (a \Psi) * \grad f
			= \grad ((a \Psi) * f).
	\end{align*}
\end{remark}

\subsection{\HolderZygmund spaces}\label{S:Holder}
We now introduce the Littlewood-Paley-based version of \Holder (more properly \HolderZygmund) spaces.
\begin{definition}\label{holderspaces}
For $r\in\R$, we define $C^r(\R^d)$ to be the set of all $f\in \mathcal{S}'(\R^d)$ such that
\begin{equation*}
\sup_{j\geq -1} 2^{jr}\| \Delta_j f \|_{L^{\infty}} <\infty.
\end{equation*} 
We set
\begin{equation*}
\| f \|_{C^r} = \sup_{j\geq -1} 2^{jr}\| \Delta_j f \|_{L^{\infty}}.
\end{equation*} 
\end{definition}
It is well-known that when $r>0$ is a non-integer, the space $C^r(\R^d)$ defined above coincides with the classical \Holder space $\tilde{C}^r(\R^d)$, with norm
\begin{equation}
\| f \|_{\tilde{C}^r} = \sum_{0\leq |\alpha| \leq [r]} \| D^{\alpha} f \|_{L^{\infty}} + \sup_{x\neq y} \frac{|f(x) - f(y)|}{|x-y|^{r-[r]}}.
\end{equation}
However, when $r$ is an integer, $C^r(\R^d)$ does not coincide with the space $\tilde{C}^r(\R^d)$ of bounded functions with bounded derivatives up to and including order $r$.  In this case, we have the inclusion
\begin{equation*}
\tilde{C}^r(\R^d) \subset C^r(\R^d).
\end{equation*}      
Finally, we define the homogeneous \Holder spaces.
\begin{definition}\label{homogeneousholderspaces}
For $r\in\R$, we define $\dot{C}^r(\R^d)$ to be the set of all $f\in \mathcal{S}'(\R^d)$ such that
\begin{equation*}
\sup_{j\in\Z} 2^{jr}\| \dot{\Delta}_j f \|_{L^{\infty}} <\infty.
\end{equation*} 
We set
\begin{equation*}
\| f \|_{\dot{C}^r} = \sup_{j\in\Z} 2^{jr}\| \dot{\Delta}_j f \|_{L^{\infty}}.
\end{equation*} 
\end{definition}
The homogeneous Littlewood-Paley operators and \HolderZygmund spaces $\dot{C}^r(\R^d)$ will be useful in our analysis of nondecaying solutions to ($SQG$) and ($E$).  In particular, the operators $\dot{\Delta}_j$ allow us to make sense of the Riesz transforms applied to non-decaying functions by defining, for $f\in L^{\infty}(\R^d)$, 
\begin{equation}\label{LPRieszTransform}
\dot{\Delta}_j \partial_k(-\Delta)^{-1/2} f = \mathcal{F}^{-1} \left(\hat{\varphi}_j \frac{i\xi_k}{|\xi|}\hat{f}\right) = \mathcal{F}^{-1} \left(\hat{\varphi}_j \frac{i\xi_k}{|\xi|} \right)\ast f.
\end{equation}

The following lemmas will be useful when proving estimates on ($SQG$) in the $C^r$ spaces.
\begin{lemma}\label{holderrieszbound}
Let $s>1$.  If for every $j\geq 0$, $f\in L^{\infty}(\R^d)$ and $g\in C^s(\R^d)$ satisfy
\begin{equation*}
{\Delta}_j f = {\Delta}_j\nabla^{\perp}(-\Delta)^{-1/2}g
\end{equation*}
almost everywhere on $\R^d$, then $f$ belongs to $C^s(\R^d)$, and there exists an absolute constant $C>0$ such that
\begin{equation*}
\| f \|_{C^s} \leq C(\| f \|_{L^{\infty}} + \| g \|_{C^s}).
\end{equation*}
\end{lemma}
\begin{proof}
Young's inequality gives 
\begin{equation*}
\begin{split}
&\| f \|_{C^s} \leq C\| \Delta_{-1} f \|_{L^{\infty}} + \sup_{j\geq 0} 2^{js}\| \Delta_{j}f \|_{L^{\infty}}\\
&\qquad \leq C\| f\|_{L^{\infty}} + \sup_{j\geq 0} 2^{js}\|  \nabla^{\perp}(-\Delta)^{-1/2}\Delta_{j}g \|_{L^{\infty}}\\
&\qquad \leq C\| f \|_{L^{\infty}} + C\sup_{j\geq 0} 2^{js}\| \Delta_{j} g \|_{L^{\infty}}\\
&\qquad \leq C\| f \|_{L^{\infty}} + C\| g\|_{C^s},
\end{split}
\end{equation*}
where we used Lemma \ref{L:CZLike} to get the third inequality.  
\end{proof}
The following Lemma is Proposition 2.2 of \cite{Wu}.
\begin{lemma}\label{holderzygmund}
Let $k$ be a nonnegative integer and let $s\in (0,1)$.  For $f\in C^{k+s}(\R^d)$, there exists a constant $C$, depending only on $s$, such that
\begin{equation*}
\| f \|_{\tilde{C}^k} \leq C \| f \|_{C^{k+s}}.
\end{equation*}
Moreover, $C\rightarrow \infty$ as $s\rightarrow 0$.
\end{lemma}

\begin{lemma}\label{L:WasSmallConst}
Let $s>0$, and assume $f\in C^s(\R^2)$.  Then 
\begin{equation*}
\|\nabla^{\perp}(a_{\lambda}\Phi) \ast f \|_{L^{\infty}} \leq C \| f \|_{C^s},
\end{equation*}
where $C$ depends only on $\lambda$ and $s$.%
\end{lemma}

\begin{proof}
Write 
\begin{equation*}
\begin{split}
&\|\nabla^{\perp}(a_{\lambda}\Phi) \ast f \|_{L^{\infty}} \leq  \sum_{j\geq -1}\|\Delta_{j} ( \nabla^{\perp}(a_{\lambda}\Phi)  \ast  f) \|_{L^{\infty}} \\
&\qquad = \|\Delta_{-1} ( \nabla^{\perp}(a_{\lambda}\Phi)  \ast  f) \|_{L^{\infty}}   + \sum_{j\geq  0} 2^{js}2^{-js}\| \Delta_{j} ( \nabla^{\perp}(a_{\lambda}\Phi)  \ast  f)  \|_{L^{\infty}}\\
&\qquad \leq  \| a_{\lambda}\Phi  \ast  (\Delta_{-1}\nabla^{\perp}f) \|_{L^{\infty}}   + C\sup_{j\geq  0} 2^{js}\|   \nabla^{\perp}(a_{\lambda}\Phi)  \ast \Delta_{j} f  \|_{L^{\infty}}\\
&\qquad \leq C\| f \|_{L^{\infty}} + C\sup_{j\geq  0} 2^{js}\|  \Delta_{j} f  \|_{L^{\infty}} \leq C \| f \|_{C^{s}},
\end{split}
\end{equation*}
where we used Young's inequality, Bernstein's Lemma and Lemma \ref{L:CZLike} to get the third inequality.  This proves the lemma.

\end{proof}

\subsection{Uniformly local Sobolev spaces}\label{S:Hsul}
We now define the uniformly local Sobolev spaces and mention some of their properties.  We refer the reader to \cite{Kato} for further details.  We begin with a definition of $L^p_{ul}(\R^d)$.
\begin{definition}
For $p\in [1, \infty)$, we define $L^p_{ul}(\R^d)$ to be the set of all functions $f$ on $\R^d$ such that
\begin{equation}\label{Lpulnorm}
\| f \|_{L^p_{ul}} := \sup_{x\in\R^d} \left( \int_{|x-y|<1} |f(y)|^p \, dy \right)^{1\slash p} < \infty. 
\end{equation}
\end{definition}
\begin{definition}
For a nonnegative integer $s$, we define the space $H^s_{ul}(\R^d)$ to be the set of all functions $f\in L^{2}_{ul}(\R^d)$ such that all distributional derivatives $D^{\alpha}f$, with $|\alpha| \leq s$, also belong to $L^2_{ul}(\R^d)$.  We set
\begin{equation}\label{Hsulnorm}
\| f \|_{H^s_{ul}} = \sum_{|\alpha| \leq s} \| D^{\alpha}f \|_{L^2_{ul}}.
\end{equation}   
\end{definition}
In what follows, we make use of an equivalent norm to (\ref{Hsulnorm}), as given in Proposition \ref{equivnorms} below.  For this purpose, throughout the paper we let $\phi\in C^{\infty}_c(\R^d)$ be a standard bump function, identically 1 on $B_1(0)$, with support contained in $B_2(0)$, and we set $$\phi_x(y) = \phi(y-x).$$

We have the following proposition (see, for example, \cite{Kato}).
 \begin{prop}\label{equivnorms}
One can define an equivalent norm to (\ref{Lpulnorm}) on $L^p_{ul}(\R^d)$ by   
\begin{equation*}
\sup_{x\in\R^d} \| \phi_x f \|_{L^p}.  
\end{equation*}
Moreover, if for $\lambda>0$ fixed, 
\begin{equation}\label{bumpnotation}
\phi_{x,\lambda}(y) = \phi\left(\frac{y-x}{\lambda}\right),
\end{equation}
then for any pair $\lambda_1$, $\lambda_2>0$, the two norms  
\begin{equation*}
\sup_{x\in\R^d} \| \phi_{x, \lambda_1} f \|_{L^p}, \quad \sup_{x\in\R^d} \| \phi_{x, \lambda_2} f \|_{L^p} 
\end{equation*}
are equivalent.  Therefore, for any $\lambda>0$, the norm 
\begin{equation}\label{Hsullambda}
\| f \|_{H^s_{ul, \lambda}} :=\sum_{|\alpha| \leq s} \sup_{x\in\R^d}\| \phi_{x,\lambda}D^{\alpha}f \|_{L^2}
\end{equation} 
is equivalent to that in (\ref{Hsulnorm}).  Finally, the norm 
$$\sup_{x\in\R^d}\| \phi_{x,\lambda} f \|_{H^s}$$
is equivalent to that in (\ref{Hsullambda}) and can also be used as a norm on $H^s_{ul}(\R^d)$.   
\end{prop}

We now state a few useful lemmas regarding $H^s_{ul}$ spaces.  Several of these lemmas demonstrate that many properties of $H^s$ spaces extend to the $H^s_{ul}$ spaces.  We begin with the following Calculus inequalities.   Parts $(i)$ and $(iii)$ below can be found in \cite{Majda}.
\begin{lemma}\label{commutator}
Assume $s \ge 1$ is an integer.

\noindent (i)  Given $f$, $g\in H^s_{ul} \cap L^{\infty}(\R^d)$ and $|\alpha|\leq s$,
\begin{equation*}
\| D^{\alpha}(fg) \|_{L^2_{ul}} \leq C_{s} (\| f \|_{L^{\infty}}\| g \|_{H^{s}_{ul}}  + \| g \|_{L^{\infty}}\| f \|_{H^{s}_{ul}} ).  
\end{equation*} 
(ii) Given $f \in \tilde{C}^s(\R^d)$, $g \in H^s(\R^d)$,
\begin{equation*}
	\norm{fg}_{H^s} \le C \smallnorm{{f}}_{\tilde{C}^s} \norm{g}_{H^s(\supp f)}, \quad
	\norm{fg}_{H^s_{ul}} \le C \smallnorm{{f}}_{\tilde{C}^s} \norm{g}_{H^s_{ul}}.
\end{equation*}
(iii)  Given $f\in H^s_{ul}\cap \tilde{C}^1(\R^d)$ and $g\in H^{s-1}_{ul}\cap L^{\infty}(\R^d)$, for $|\alpha| \leq s$,
\begin{equation*}
\| D^{\alpha}(fg) - fD^{\alpha}g \|_{L^2_{ul}} \leq C_{s} (\| f \|_{\tilde{C}^1}\| g \|_{H^{s-1}_{ul}}  + \| g \|_{L^{\infty}}\| f \|_{H^{s}_{ul}} ).  
\end{equation*}
\end{lemma}
\begin{lemma}\label{SobolevEmbedding}
(\cite{Kato})  Let $j$ and $m$ be nonnegative real numbers.  If $2m>d$, then $H^{j+m}_{ul}(\R^d)\hookrightarrow \tilde{C}^j(\R^d)$. 
\end{lemma}

\begin{lemma}\label{smoothcutoffbounded}
    Let $p\in[1,\infty)$, and assume $f$ belongs to $L_{ul}^p(\R^d)$. There exists $C>0$ such that for all $n\in\N$,
    \[\|S_n f\|_{L^p_{ul}} \leq C\|f\|_{L^p_{ul}}.\] 
\end{lemma}
\begin{proof}
    By Minkowski's inequality, 
    $$\|S_n f\|_{L^p_{ul}}=\sup\limits_{z\in\R^d}\Bigg(\int\limits_{\R^d}\Big|\int\limits_{\R^d}\phi_z(x) f(x-y) \chi_n(y) dy\Big|^pdx\Bigg)^{1/p}$$
    $$\leq \sup\limits_{z\in\R^d}\int\limits_{\R^d}\|\phi_z(\cdot) f(\cdot-y) \chi_n(y)\|_{L^p}dy =  \sup\limits_{z\in\R^d}\int\limits_{\R^d}\|\phi_z(\cdot) f(\cdot-y) \|_{L^p} |\chi_n(y)| dy$$
    $$\leq \|f\|_{L^p_{ul}}\int_{\R^d} |\chi_n(y)|dy\leq C\|f\|_{L^p_{ul}}.$$ 
\end{proof}

\begin{lemma}\label{L:ConvHsBound}
	With $\Psi$ as in Lemma \ref{L:CZLike},
	for any $f \in H^r_{ul}(\R^d)$, $r \ge 0$,
	\begin{align*}
		\norm{\grad ((a_\lambda \Psi) * f)}_{H^r_{ul}}
			&\le C \lambda \norm{f}_{H^r_{ul}}.
	\end{align*}
\end{lemma}
\begin{proof}
	This follows for $H^r$ in place of $H^r_{ul}$
	from a Littlewood-Paley decomposition or by using the expression
	in Lemma \ref{L:CZLike}. It then follows for the $H^r_{ul}$ norm
	by taking advantage of the identity, 
	$\phi_x \grad ((a_\lambda \Psi) * f) =
	\phi_x (\nabla(a_{\lambda} \Psi) \ast f) = \phi_x (\nabla(a_{\lambda} \Psi) \ast \phi_{x,8} f)$.
\end{proof}

\begin{definition}\label{D:stardot}
	For $v$, $w$ vector fields, we define
	$
		v \stardot w
			= v^i * w^i,
	$
	where we sum over the repeated indices. Similarly,
	for $A$, $B$ matrix-valued functions on $\R^d$, we define
	$
		A \stardot B
			= A^{ij} * B^{ij}.
	$
\end{definition}

In Lemma \ref{L:Stream}, we obtain a \textit{stream function} for $\psi$, but it is not the classical stream function in that it is not divergence-free. It can be written in the form of a one-dimensional integral, however, as in (\ref{e:StreamSimplified}), which makes it amenable to localized estimates.

\begin{lemma}\label{L:Stream}
	For any divergence-free $u \in H^s_{ul}(\R^3)$ there exists a
	(non-divergence free)
	\textit{stream function} $\psi \in H^{s + 1}_{ul}(\R^3)$ with
	the properties that $\curl \psi = u$, 
	$\psi(0) = 0$.
	For any bounded convex $U \subseteq B_R(0)$,
	\begin{align}\label{e:StreamBound}
		\norm{\psi}_{H^s(U)}
			\le C R \norm{u}_{H^s_{ul}(\R^3)},
	\end{align}
	where the constant $C$ depends upon the Lebesgue measure, $\abs{U}$, of $U$.
\end{lemma}
\begin{proof}
It is sufficient to prove the result for $u \in C^\iny(\R^3) \cap H^s_{ul}(\R^3)$, as the result then follows from the density of this space in $H^s_{ul}(\R^3)$. We can then define the stream function as
\begin{align}\label{e:StreamSimplified}
	\psi(x)
		&= - \int_0^1 \tau x \times u(\tau x) \, d \tau.
\end{align}
Using $\curl (A \times B) = \dv B \, A - \dv A \, B + B \cdot \grad A - A \cdot \grad B$, $\dv (u(\tau x)) = 0$, $\dv x = 3$, $\grad x = I$, we have
\begin{align*}
	\curl (&x \times u(\tau x))
		= - 3 u (\tau x)  + u (\tau x)  \cdot \grad x - \tau x \cdot \grad u (\tau x)  \\
		&= - 3 u (\tau x)  + u (\tau x)  \cdot I - \tau x \cdot \grad u (\tau x) 
		= - 2 u (\tau x) - \tau x \cdot \grad u (\tau x). 
\end{align*}
Hence,
\begin{align*}
	\curl \psi(x)
		&= \int_0^1
			\brac{2 \tau u (\tau x) + \tau^2 x \cdot \grad u (\tau x)}  d \tau.
\end{align*}
Integrating the first term by parts, we have
\begin{align*}
	&\int_0^1 2 \tau u (\tau x) \, d \tau
		= \tau^2 u(\tau x)\big\vert_0^1
			- \int_0^1 \tau^2 x \cdot \grad u(\tau x) \, d \tau\\
		&\qquad= u(x) - \int_0^1 \tau^2 x \cdot \grad u(\tau x) \, d \tau.
\end{align*}
It follows that $\curl \psi = u$.

For estimates, it is perhaps easier to write (\ref{e:StreamSimplified}) in indices, as
\begin{align}\label{e:Stream}
    \psi^i(x)=\int_0^1 \big[\tau x_{i+2} u^{i+1}(\tau x)
    	-\tau x_{i+1}u^{i+2}(\tau x)\big]\, d \tau,
\end{align}
where if $i + 1$ or $i + 2 > 3$ we subtract 3 from it.

In (\ref{e:Stream}), we have $\abs{x} \le R$ on $U$, so
\begin{align*}
	\norm{\psi}_{L^2(U)}
		&\le C R \int_0^1 \tau \smallnorm{u(\tau \cdot)}_{L^2(U)} \, d \tau
		= C R \int_0^1 \frac{\tau}{\tau^{\frac{3}{2}}} \smallnorm{u}_{L^2(\tau U)} \, d \tau.
\end{align*}
But, $\abs{\tau U} \le \abs{U}$ for all $\tau \in [0, 1]$, so
$
	\smallnorm{u}_{L^2(\tau U)}
		\le C(\abs{U}) \smallnorm{u}_{L^2_{ul}}
$
and
\begin{align*}
	\norm{\psi}_{L^2(U)}
		&\le C R 
			\int_0^1 \tau^{-\frac{1}{2}} \smallnorm{u}_{L^2_{ul}} \, d \tau
		= C R.
\end{align*}
		
Let $\al = (\al_1, \al_2, \al_3)$ be a multi-index. Then 
\begin{align*}	
	D^\alpha \brac{\tau x_j u^\ell (\tau x)}
		&= \tau x_j \tau^{\abs{\al}} D^\al u^\ell (\tau x)
			+ \tau \tau^{\abs{\al} - 1} D^{\al'} u^\ell (\tau x),
\end{align*}
where $\al'$ has the $j$ index decreased by one, with the second term absent if $\al_j = 0$. Arguing as for $\norm{\psi}_{L^2(U)}$, we conclude from this that
\begin{align*}
	\sup_{\abs{\al} = k} \norm{D^\al \psi}_{L^2(U)}
		&\le C R \norm{u}_{H^k_{ul}} + C \norm{u}_{H^{k - 1}_{ul}},
\end{align*}
from which (\ref{e:StreamBound}) follows by summing over $k \le s$.

Finally, using $\dv (A \times B) = (\curl A) \cdot B - (\curl B) \cdot A$, we have
\begin{align*}
	\dv \psi(x)
		&= - \int_0^1 \tau
			\brac{\curl x \cdot u(\tau x) - \tau (\curl u)(\tau x) \cdot x} \, d \tau \\
		&= \int_0^1 \tau^2 x \cdot (\curl u)(\tau x) \, d \tau.
\end{align*}
Because $\curl u \in L^2_{ul}$, $\curl \psi$ and $\dv \psi$ lie in $L^2_{ul},$ this is enough to conclude that $\psi \in H^{s + 1}_{ul}(\R^3).$  As we use only $H^s_{ul}$ regularity, we do not include further details.
\end{proof}

\section{Existence of solutions to ($SQG$) in \Holder spaces}\label{S:SQGHolder}
In this section, we prove the following theorem.
\begin{theorem}\label{SQGHolder}
For $r\in (1,\infty)$, let $\theta^0$ be a function in $C^{r}(\R^2)$, and let $u^0$ in $C^r(\R^2)$ satisfy
\begin{equation*}
u^0 =  \nabla^{\perp} (-\Delta)^{-1/2}  {\theta}^0 \text{  in }\dot{C}^r(\R^2).
\end{equation*}
There exists $T>0$ and a unique solution $(u, \theta)$ to 
\begin{equation}\label{SQG}
\begin{split}
&\partial_t \theta + u\cdot\nabla\theta = 0,\\
&(u, \theta)|_{t=0} = (u^0, \theta^0),
\end{split}
\end{equation}   
satisfying, for any $r' \in (0, r)$,
\begin{equation*}
\begin{split}
&\theta\in L^{\infty}(0,T; C^{r}(\R^2)) \cap Lip([0,T] ; C^{r-1}(\R^2)) \cap C([0,T]; C^{r'}(\R^2)),\\
&u \in L^{\infty}(0,T; C^{r}(\R^2)) \cap C([0,T]; C^{r'}(\R^2)).  
\end{split}
\end{equation*}    

Moreover, there exists $C>0$ such that $(u,\theta)$ satisfies the estimate
\begin{equation}\label{shorttimebound}
\| u \|_{L^{\infty}(0,T; L^{\infty})} + \| \theta \|_{L^{\infty}(0,T; C^r)} \leq \frac{C(\| u^0 \|_{L^{\infty}} + \| \theta^0 \|_{C^r})}{1-CT(\| u^0 \|_{L^{\infty}} + \| \theta^0 \|_{C^r})}
\end{equation}
and the equality (see Definition \ref{D:stardot})
\begin{equation}\label{MainThmSerfatiSQG}
\begin{split}
&u(t) = u^0
+(a\Phi) * \nabla^{\perp}(\theta(t) - \theta^0)
- \int_0^t (\nabla\nabla^\perp((1-a)\Phi)) \stardot (\theta u)
\end{split}
\end{equation}
for each $t\in[0,T]$.
\end{theorem}  
Before proving the theorem, we make a few  remarks.
\begin{remark}
For $r>0$ a non-integer, a pair $(u^0, \theta^0)$ satisfying the conditions of Theorem \ref{SQGHolder} can be easily generated from any function $\psi\in C^{r+1}(\R^2)$ by setting $u^0 = \nabla^{\perp}\psi$ and $\omega^0 = (-\Delta)^{1/2}\psi$.  Note that $\omega^0$ belongs to $C^r(\R^2)$ by the classical Schauder estimates for the fractional Laplacian (see for example \cite{Stinga}).  By the containment ${C}^r(\R^2) \subset \dot{C}^r(\R^2)$ and Lemma \ref{L:CZLike}, both $u^0$ and $\theta^0$ belong to $\dot{C}^r(\R^2)$.  Moreover, we have for every  $j\in\Z$, $$\dot{\Delta}_ju^0 =  \dot{\Delta}_j\nabla^{\perp} (-\Delta)^{-1/2}  {\theta}^0$$ almost everywhere on $\R^2$.   
\end{remark}

\begin{remark}
	Since (\ref{MainThmSerfatiSQG}) holds for $a_\lambda$ in place of $a$ for any
	$\lambda > 0$, and $\theta$ and $u$ lie in $L^\iny([0, T] \times \R^2)$,
	by (\ref{e:GPhiBounds}) we have
	\begin{align*}
		u(t) = u^0
			+ \lim_{\lambda \to \iny} (a_\lambda \Phi) \ast \nabla^{\perp} (\theta(t) - \theta^0),
	\end{align*}
	the limit holding pointwise. This gives a form of the constitutive law for (\ref{SQG})
	and is the analog for ($SQG$) of the renormalized Biot-Savart law of \cite{AKLN,Kelliher}
	that applies to non-decaying solutions to the 2D Euler equations.
\end{remark}

\begin{proof}[\textbf{Proof of Theorem \ref{SQGHolder}}]
We adapt the general strategy used in the proof of Theorem 4.1 in \cite{Wu}.  In particular, we construct an approximating sequence of solutions and pass to the limit in the appropriate norm.  To obtain uniform bounds on the approximating sequence, the proof in \cite{Wu} relies heavily on the estimate
\begin{equation}\label{WuRiesz}
\| Rf \|_{C^r} \leq C\| f \|_{C^r\cap L^q}
\end{equation}
for $q<\infty$ and $r>1$, where $R$ denotes a Riesz transform.  Since our approximating sequence must converge to a solution lacking spatial decay (and hence not belonging to $L^q(\R^2)$ for any $q<\infty$), we utilize Lemma \ref{holderrieszbound} and a Serfati-type identity (see (\ref{seqapprox}) below) in place of (\ref{WuRiesz}).     
   
\bigskip
\noindent {\bf Approximating sequence.}  We define sequences $(\theta^n)_{n=1}^{\infty}$ and $(u^n)_{n=1}^{\infty}$ as follows:
\begin{equation}\label{iteration}
\begin{split}
&\theta^1(t,x) = S_2\theta^0(x),\\
&u^1(t,x) = S_2u^0(x),
\end{split}
\end{equation}
for all $t\geq 0$, while, for $n\geq 1$,
\begin{align}\label{seqapprox}
\begin{split}
&\partial_t\theta^{n+1} + u^n\cdot\nabla\theta^{n+1} =0,\\
&\theta^{n+1} (x,0) = S_{n+2}\theta^0, u^{n+1} (x,0) = S_{n+2}u^0,\\
&u^{n+1}(t)= u^{n+1}(0)+  (a\Phi) * \grad^\perp(\theta^{n+1}(t) - \theta^{n+1}(0))\\
&\qquad - \int_0^t (\grad \grad^\perp((1 - a) \Phi)) \stardot (\theta^{n+1} u^{n}).
\end{split}
\end{align}
Note that with $(u^n)$ and $(\theta^n)$ as in (\ref{seqapprox}), Lemma \ref{BMOSerfati} gives that for all $j\in\Z$, $n\in\N$, and $t\in[0,T]$,
$$ \dot{\Delta}_j u^n(t) = \dot{\Delta}_j \nabla^{\perp}(-\Delta)^{-1/2}\theta^n(t)$$
almost everywhere on $\R^2$, which will allow us to apply Lemma \ref{holderrieszbound} repeatedly in what follows.

\bigskip

\noindent {\bf Uniform Bounds.}  The proof of Proposition 4.2 in \cite{Wu} yields the following estimate:
\begin{equation}\label{scalarholder}
\begin{split}
&\| \theta^{n+1} (t) \|_{C^r} \leq \| \theta^{n+1} (0) \|_{C^r}\\
& + C(r)\int_0^t \left( \| \nabla \theta^{n+1} (s) \|_{L^{\infty}} \| u^n(s) \|_{C^r} + \| \nabla u^n(s) \|_{L^{\infty}} \| \theta^{n+1}(s) \|_{C^r}  \right) \, ds\\
& \leq \| \theta^{n+1} (0) \|_{C^r} + C(r)\int_0^t \| \theta^{n+1} (s) \|_{C^r} \| u^n(s) \|_{C^r}    \, ds.
\end{split}
\end{equation}
We now use $(\ref{seqapprox})_3$ to estimate $\| u^{n+1}(t) \|_{L^{\infty}}$.  In particular, one can write
\begin{equation}\label{velinf}
\begin{split}
&\| u^{n+1}(t) \|_{L^{\infty}} \leq \| u^{n+1}(0) \|_{L^{\infty}} + C\| \nabla{\theta}^{n+1}(t) \|_{L^{\infty}}\\
&\qquad  + C\| \nabla{\theta}^{n+1}(0) \|_{L^{\infty}} + C\int_0^t \| {\theta}^{n+1}(s) \|_{L^{\infty}}\| u^{n}(s) \|_{L^{\infty}} \, ds.
\end{split}
\end{equation}
Adding (\ref{scalarholder}) and (\ref{velinf}) gives
\begin{equation*}
\begin{split}
&\| u^{n+1}(t) \|_{L^{\infty}} + \| \theta^{n+1} (t) \|_{C^r}  \leq \| u^{n+1}(0) \|_{L^{\infty}} + C(r)\| {\theta}^{n+1}(t) \|_{C^r} + C(r)\| {\theta}^{n+1}(0) \|_{C^r}\\
&\qquad+ C(r)\int_0^t (\| {\theta}^{n+1}(s) \|_{L^{\infty}}\| u^{n}(s) \|_{L^{\infty}} + \| \theta^{n+1} (s) \|_{C^r} \| u^n(s) \|_{C^r}) \, ds,
\end{split}
\end{equation*}
where we used Lemma \ref{holderzygmund}.  The term $C(r)\| {\theta}^{n+1}(t) \|_{C^r}$ appearing on the right hand side can again be estimated using (\ref{scalarholder}).  Then we have
\begin{equation}\label{shorttimeboundprelimit}
\begin{split}
&\| u^{n+1}(t) \|_{L^{\infty}} +  \| \theta^{n+1} (t) \|_{C^r}  \leq \| u^{n+1}(0) \|_{L^{\infty}} + C(r) \| {\theta}^{n+1}(0) \|_{C^r}\\
&\qquad + C(r)\int_0^t  \| \theta^{n+1} (s) \|_{C^r} \| u^n(s) \|_{C^r} \, ds\\
&\leq \| u^{n+1}(0) \|_{L^{\infty}} + C(r) \| {\theta}^{n+1}(0) \|_{C^r} \\
&\qquad+ C(r)\int_0^t  (\| u^{n+1}(s) \|_{L^{\infty}} + \| \theta^{n+1} (s) \|_{C^r}) \| u^n(s) \|_{C^r} \, ds.
\end{split}
\end{equation}
\Gronwalls Lemma gives
\begin{equation}\label{1}
\| u^{n+1}(t) \|_{L^{\infty}} + \| \theta^{n+1} (t) \|_{C^r}  \leq C(r)(\| u^{n+1}(0) \|_{L^{\infty}} + \| {\theta}^{n+1}(0) \|_{C^r})e^{C(r)\int_0^t \| u^n(s) \|_{C^r}}.
\end{equation}  
By (\ref{1}) and Lemma \ref{holderrieszbound},
\begin{equation}\label{velocitypostGronwall}
\begin{split}
&\| u^{n+1}(t) \|_{L^{\infty}} + \| \theta^{n+1} (t) \|_{C^r} \\
&\qquad \leq C(r)(\| u^{n+1}(0) \|_{L^{\infty}} + \| {\theta}^{n+1}(0) \|_{C^r})e^{C(r)\int_0^t (\| u^n(s) \|_{L^{\infty}} + \| \theta^{n} (s) \|_{C^r}) \, ds},
\end{split}
\end{equation}   
where we can assume $C(r)\geq 2$.  

We use induction and (\ref{velocitypostGronwall}) to show that there exists $M>0$ and $T>0$ such that, for all $t<T$, and for all $n\geq 1$,
\begin{equation}\label{induction}
\| u^{n}(t) \|_{L^{\infty}} + \| \theta^{n} (t) \|_{C^r}  \leq M.
\end{equation}  
To prove the case $n=1$, first note that by properties of Littlewood-Paley operators and Young's inequality, there exists a constant $C_2$ such that, for all $n\geq 1$,
\begin{equation}\label{iditerationbound}
\| S_{n+2} u^0 \|_{L^{\infty}}  +  \| S_{n+2} \theta^0 \|_{C^r} \leq C_2( \|  u^0 \|_{L^{\infty}}  +  \| \theta^0 \|_{C^r}). 
\end{equation}
In particular, we have
\begin{equation*}
\| u^1 \|_{L^{\infty}} + \| \theta^1 \|_{C^r} \leq C_2(\| u^0 \|_{L^{\infty}} + \| {\theta}^0 \|_{C^r}).
\end{equation*}
Set $M = 2C(r)C_2(\| u^0 \|_{L^{\infty}} + \| {\theta}^0 \|_{C^r})$, where $C(r)$ is as in (\ref{velocitypostGronwall}), and choose $T$ such that exp$(C(r)TM) \leq 2$.  Then
\begin{equation*}
\| u^1 \|_{L^{\infty}} + \| \theta^1 \|_{C^r} \leq C_2(\| u^0 \|_{L^{\infty}} + \| {\theta}^0 \|_{C^r})<M.
\end{equation*}
This proves (\ref{induction}) for $n=1$.  

Now assume, for fixed $k\in\N$, $\| u^k(s) \|_{L^{\infty}} + \| \theta^{k} (s) \|_{C^r} \leq M$ for each $s\in [0,T]$.  By (\ref{velocitypostGronwall}) and (\ref{iditerationbound}),
\begin{equation*}
\begin{split}
&\| u^{k+1}(t) \|_{L^{\infty}} + \| \theta^{k+1} (t) \|_{C^r} \leq C(r)(\| u^{k+1}(0) \|_{L^{\infty}} + \| {\theta}^{k+1}(0) \|_{C^r})e^{C(r)TM}\\
&\qquad \leq 2C(r)C_2(\| u(0) \|_{L^{\infty}} + \| {\theta}(0) \|_{C^r}) = M.
\end{split}
\end{equation*}
Thus (\ref{induction}) holds for all $n$.

From (\ref{induction}) and Lemma \ref{holderrieszbound}, it follows that, for $r>1$, there exists $C>0$ such that for all $n\in\N$, $\| u^{n} \|_{C^r} \leq CM$.  Therefore, for each $n\in\N$,
\begin{equation}\label{timederivativebound}
\begin{split}
&\| \partial_t\theta^{n+1} \|_{C^{r-1}} \leq \| u^n\cdot\nabla\theta^{n+1} \|_{C^{r-1}}\\
&\qquad \leq C(\| u^n\|_{C^{r-1}}\|\nabla\theta^{n+1} \|_{L^{\infty}} +  \| u^n\|_{L^{\infty}}\|\nabla\theta^{n+1} \|_{C^{r-1}})\\
&\qquad \leq C(r)\| u^n\|_{C^{r-1}}\|\theta^{n+1} \|_{C^{r}} \leq C(r)M^2.
\end{split} 
\end{equation}
From this we conclude that for each $n\in\N$, $\partial_t\theta^{n}\in L^{\infty}(0,T; C^{r-1})$ and $\theta^n \in Lip([0,T]; C^{r-1})$, with norms uniformly bounded in $n$.\\
\\
{\bf $\bm{(u^n)}$ and $\bm{(\theta^n)}$ are Cauchy.}  We now show $(\theta^n )$ is Cauchy in $C([0,T]; C^{r-1}(\R^2))$ and $( u^n )$ is Cauchy in $C([0,T]; L^{\infty}(\R^2))$.  As in \cite{Wu}, let $\eta^n = \theta^{n} - \theta^{n-1}$ and $v^{n} = u^{n} - u^{n-1}$.  From (\ref{iteration}) and (\ref{seqapprox}), we have the system
\begin{equation}\label{iteration1}
\begin{split}
&\eta^1 = S_{2}\theta^0 - \theta^0,\\
&v^1 = S_{2}u^0 - u^0,
\end{split}
\end{equation}
and for $n\geq 1$,
\begin{equation}\label{seqapprox1}
\begin{split}
&\partial_t\eta^{n+1} + u^n\cdot\nabla\eta^{n+1} =v^n\cdot\nabla\theta^n,\\
&\eta^{n+1} (x,0) = \eta_0^{n+1} (x) = \Delta_{n+2}\theta^0(x),\\
&v^{n+1} (x,0) = v_0^{n+1} (x)= \Delta_{n+2}u^0(x).
\end{split}
\end{equation}
Moreover,
\begin{equation}\label{Serfatiapprox12}
\begin{split}
&v^n(t) -  v^n(0) =  (a_{\lambda} \Phi) * \grad^\perp(\eta^n(t) - \eta^n(0))\\
&\qquad - \int_0^t (\grad \grad^\perp ((1 - a_{\lambda}) \Phi)) \stardot (\eta^n u^{n-1} + \theta^{n} v^{n-1}).
\end{split}
\end{equation}  

We have the following estimate from \cite{Wu}:
\begin{equation}\label{scalardiff}
\begin{split}
&\| \eta^{n+1} (t) \|_{C^{r-1}} \leq \| \eta^{n+1} (0) \|_{C^{r-1}}\\
& + C(r)\int_0^t (\| \eta^{n+1} (s) \|_{C^{r-1}}\| u^{n} (s) \|_{C^{r}}  + \| v^n(s) \|_{C^{r-1}} \|\theta^n(s) \|_{C^r}) \, ds\\
&\qquad \leq \| \eta^{n+1} (0) \|_{C^{r-1}} + C(r)M\int_0^t (\| \eta^{n+1} (s) \|_{C^{r-1}}  + \| v^n(s) \|_{C^{r-1}} ) \, ds,
\end{split}
\end{equation}
where we applied the uniform bounds on $\| u^n \|_{C^r}$ and $\| \theta^n\|_{C^r}$ to get the second inequality.  We apply the $L^{\infty}$-norm to $(\ref{Serfatiapprox12})$, which gives
\begin{equation}\label{veldiff}
\begin{split}
&\| v^{n+1}(t) \|_{L^{\infty}} \leq \| v^{n+1}(0) \|_{L^{\infty}} + \|(a\Phi) \ast\nabla^{\perp}{\eta}^{n+1}(t) \|_{L^{\infty}} + \|(a\Phi) \ast\nabla^{\perp}{\eta}^{n+1}(0) \|_{L^{\infty}}\\
&\qquad   + \int_0^t (\| {\eta}^{n+1}(s) \|_{L^{\infty}}\| u^{n}(s) \|_{L^{\infty}} + \| {\theta}^{n+1}(s) \|_{L^{\infty}}\| v^{n}(s) \|_{L^{\infty}}) \, ds\\
&  \leq \| v^{n+1}(0) \|_{L^{\infty}} + \|(a\Phi) \ast\nabla^{\perp}{\eta}^{n+1}(t) \|_{L^{\infty}} + \|(a\Phi) \ast\nabla^{\perp}{\eta}^{n+1}(0) \|_{L^{\infty}}\\
&\qquad   + CM\int_0^t (\| {\eta}^{n+1}(s) \|_{L^{\infty}} + \| v^{n}(s) \|_{L^{\infty}}) \, ds. 
\end{split}
\end{equation}
Adding (\ref{scalardiff}) and (\ref{veldiff}) gives
\begin{equation}\label{velscalar}
\begin{split}
&\| v^{n+1}(t) \|_{L^{\infty}} + \| \eta^{n+1} (t) \|_{C^{r-1}}   \leq \| v^{n+1}(0) \|_{L^{\infty}} +  \| \eta^{n+1} (0) \|_{C^{r-1}}\\
&\qquad + \|\nabla^{\perp}(a\Phi) \ast{\eta}^{n+1}(0) \|_{L^{\infty}} + \|\nabla^{\perp}(a\Phi) \ast{\eta}^{n+1}(t) \|_{L^{\infty}}\\
& + MC(r)\int_0^t ( \| v^{n+1} \|_{L^{\infty}}  + \| \eta^{n+1} (s) \|_{C^{r-1}}  + \| v^{n}(s) \|_{L^{\infty}} + \| \eta^{n} (s) \|_{C^{r-1}}) \, ds,
\end{split}
\end{equation}
where we applied Lemma \ref{holderrieszbound}.  To estimate the terms $\|\nabla^{\perp}(a\Phi) \ast{\eta}^{n+1}(t) \|_{L^{\infty}}$ and $\|\nabla^{\perp}(a\Phi) \ast{\eta}^{n+1}(0) \|_{L^{\infty}}$, we apply Lemma \ref{L:WasSmallConst}, giving
\begin{equation*}
\begin{split}
&\|\nabla^{\perp}(a\Phi) \ast{\eta}^{n+1}(t) \|_{L^{\infty}} \leq C(r) \| {\eta}^{n+1}(t) \|_{C^{r-1}},\\
&\|\nabla^{\perp}(a\Phi) \ast{\eta}^{n+1}(0) \|_{L^{\infty}} \leq C(r) \| {\eta}^{n+1}(0) \|_{C^{r-1}}.
\end{split}
\end{equation*}
We then bound the resulting term $\| {\eta}^{n+1}(t) \|_{C^{r-1}}$ using (\ref{scalardiff}) and again apply Lemma \ref{holderrieszbound}.  Substituting the resulting estimate into (\ref{velscalar}) gives
\begin{equation}\label{Cauchyest}
\begin{split}
& \| v^{n+1}(t) \|_{L^{\infty}}  + \| \eta^{n+1} (t) \|_{C^{r-1}}  \leq C_1(r)(\| v^{n+1}(0) \|_{L^{\infty}} + \| {\eta}^{n+1}(0) \|_{C^{r-1}}) \\
& + C_1(r)M\int_0^t \left(( \| v^{n+1}(s) \|_{L^{\infty}}+\| \eta^{n+1} (s) \|_{C^{r-1}}) + (\| v^n(s) \|_{L^{\infty}}+ \| {\eta}^n(s) \|_{C^{r-1}} )\right) \, ds.
\end{split}
\end{equation}
Set $D_n (t) = \| v^n(t) \|_{L^{\infty}} + \| {\eta}^n(t) \|_{C^{r-1}}$.  Then (\ref{Cauchyest}) gives
\begin{equation}\label{forDavid1}
D_{n+1} (t) \leq C_1(r) D_{n+1} (0) + C_1(r)M \int_0^t  (D_{n+1} (s) + D_{n} (s)) \, ds.
\end{equation}

Let
\begin{align*}
	E(t)
		:= \sum_{n = 0}^\iny D_{n + 1}(t),
\end{align*}
noting that $E(0)$ is finite because $\theta^0$ and $u^0$ lie in $C^r(\R^2)$.
Summing (\ref{forDavid1}) over $n$ and using (\ref{induction}), we have that
\begin{align*}
	E(t)
		&\le C E(0) + C M t + C M \int_0^t E(s) \, ds.
\end{align*}
By \Gronwalls lemma,
\begin{align*}
	E(t)
		&\le (C E(0) + C M T) e^{CM t}.
\end{align*}
It follows that for any fixed time $t \in [0, T]$, the sequences $(u^n(t))$ and $(\theta^n(t))$ are Cauchy in $L^\iny(\R^2)$ and $C^{r - 1}(\R^2)$, respectively.

Now let $\epsilon > 0$. From (\ref{timederivativebound}), we also have uniform-in-time control on $(\prt_t \theta^n(t))$ in $C^{r - 1}(\R^2)$, so we can choose a $\delta > 0$ such that for any $s_1, s_2 \in [0, T]$,
\begin{align*}
	\abs{s_1 - s_2} < \delta
		\implies \norm{\theta^n(s_1) - \theta^n(s_2)}_{C^{r - 1}(\R^2)} < \frac{\epsilon}{3}.
\end{align*}
Let $N_1$ be an integer greater than $T/\delta$, and let $t_k = k T/N_1$, $k = 0, \dots, N_1$. Choose an integer $N_2$ (which we note depends upon $N_1$) such that for all $k$,
\begin{align*}
	m, n \ge N_2 \implies \norm{\theta^n(t_k) - \theta^m(t_k)}_{C^{r - 1}(\R^2)}
		< \frac{\epsilon}{3}. 
\end{align*}
Then by the triangle inequality, for all $t \in [0, T]$,
\begin{align*}
	m, n \ge N_2 \implies
		&\norm{\theta^n(t) - \theta^m(t)}_{C^{r - 1}(\R^2)}
			\le \norm{\theta^n(t) - \theta^n(t_k)}_{C^{r - 1}(\R^2)} \\
				&\quad
				+ \norm{\theta^n(t_k) - \theta^m(t_k)}_{C^{r - 1}(\R^2)}
				+ \norm{\theta^m(t_k) - \theta^m(t)}_{C^{r - 1}(\R^2)}
		< \epsilon,
\end{align*}
where we choose $k$ so that $\abs{t - t_k} < \delta$. This is enough to conclude that $(\theta^n)$ is Cauchy in $C([0,T];C^{r-1}(\R^2))$.

Similarly, taking the time derivative of (\ref{seqapprox})$_3$ gives uniform-in-time control on $(\prt_t u^n(t))$ in $L^\iny(\R^2)$, and we can conclude that $(u^n)$ is Cauchy in $C([0,T];L^{\infty}(\R^2))$.\\

\noindent {\bf Limit of the sequence solves (\ref{SQG}).}  We conclude that the sequence $(\theta^n)$ converges to some $\theta$ in $C([0,T]; C^{r-1}(\R^2))$, and $(u^n)$ converges to some $u$ in $C([0,T]; L^{\infty}(\R^2))$.  Moreover, we have
\begin{equation*}
\begin{split}
&\theta \in L^{\infty}(0,T_1; C^{r}(\R^2)) \cap Lip([0,T_1] ; C^{r-1}(\R^2)),\\
&u \in L^{\infty}(0,T_1; C^{r}(\R^2)). 
\end{split}
\end{equation*} 
Interpolation between $C^{r - 1}$ and $C^r$ shows that $(\theta^n)$ converges to $\theta$ in $C([0,T]; C^{r'}(\R^2))$ for all $r'\in [r-1,r)$, and interpolation between $C^0$ and $C^r$ shows that $(u^n)$ converges to $u$ in $C([0,T]; C^{\alpha}(\R^2))$ for all $\alpha\in (0,r)$.  Having established convergence in these spaces, we can then pass to the limit in $(\ref{seqapprox})_1$ and $(\ref{seqapprox})_2$.  Note also that, for $r'\in [r-1,r)$ and $\alpha\in(0,r)$,
\begin{equation*}
\begin{split}
&\theta\in L^{\infty}([0,T]; C^{r}(\R^2)) \cap Lip([0,T] ; C^{r-1}(\R^2)) \cap C([0,T]; C^{r'}(\R^2)),\\
&u \in L^{\infty}([0,T]; C^{r}(\R^2)) \cap C([0,T]; C^{\alpha}(\R^2)). 
\end{split}
\end{equation*}
\\
{\bf Solution $\bm{(u,\theta)}$ satisfies (\ref{shorttimebound}).}  We now show that the resulting solution ($u,\theta$) of (\ref{SQG}) satisfies (\ref{shorttimebound}).  Set $\Psi_n(\tau) = \| u^n(\tau) \|_{L^{\infty}} + \| \theta^n(\tau) \|_{C^r}$, $\tau\in[0,T]$.  From (\ref{velocitypostGronwall}), it follows that
\begin{equation*}
\Psi_n(\tau) \leq C\Psi_n(0)e^{C \int_0^\tau \Psi_n(s) \, ds},
\end{equation*}  
so that
\begin{equation*}
\Psi_n(\tau)e^{-C \int_0^\tau \Psi_n(s) \, ds} \leq C\Psi_n(0).
\end{equation*}  
By the chain rule,
\begin{equation*}
-\frac{1}{C}\frac{d}{d\tau}\left(e^{-C \int_0^\tau \Psi_n(s) \, ds}\right) \leq C\Psi_n(0).
\end{equation*}  
For $t \in[\tau, T]$, integrating both sides from $0$ to $t$  gives
\begin{equation*}
-e^{-C \int_0^t \Psi_n(s) \, ds} + 1 \leq C\Psi_n(0)t,
\end{equation*}  
which implies
\begin{equation}\label{expest}
e^{C \int_0^t \Psi_n(s) \, ds} \leq \frac{1}{1-C\Psi_n(0)t}.
\end{equation}
The inequality $\frac{\Psi_n(t)}{C\Psi_n(0)} \leq e^{C \int_0^t \Psi_n(s) \, ds}$, combined with (\ref{expest}), imply that
\begin{equation*}
\| u^n(t) \|_{L^{\infty}} + \| \theta^n(t) \|_{C^r}=\Psi_n(t) \leq \frac{C\Psi_n(0)}{1-C\Psi_n(0)t} \leq \frac{C(\| u^0 \|_{L^{\infty}} + \| \theta^0 \|_{C^r})}{1-Ct(\| u^0 \|_{L^{\infty}} + \| \theta^0 \|_{C^r})},
\end{equation*}   
where we used that $\Psi_n(0) \leq C(\| u^0 \|_{L^{\infty}} + \| \theta^0 \|_{C^r})$ for all $n$ to get the second inequality.  It follows that for each fixed $t\in [0,T]$,
\begin{equation*}
\begin{split}
&\| u(t) \|_{L^{\infty}} + \| \theta(t) \|_{C^r} \leq \frac{C(\| u^0 \|_{L^{\infty}} + \| \theta^0 \|_{C^r})}{1-Ct(\| u^0 \|_{L^{\infty}} + \| \theta^0 \|_{C^r})}.
\end{split}
\end{equation*}   
This yields (\ref{shorttimebound}). \\
\\
{\bf $\bm{(u,\theta)}$ satisfies (\ref{MainThmSerfatiSQG}).}  It remains to prove (\ref{MainThmSerfatiSQG}).  We have that $(\theta^n)$ converges to $\theta$ in $C([0,T]; C^{r'}(\R^2))$ for all $r'<r$, and $(u^n)$ converges to $u$ in $C([0,T];L^{\infty}(\R^2))$.  We claim that this is enough to pass to the limit in $(\ref{seqapprox})_3$.  For $n\in\N$, we subtract the right-hand side of $(\ref{seqapprox})_3$ as satisfied by $(u,\theta)$ from the right-hand-side as satisfied by $(u^n,\theta^n)$.  Taking the $L^{\infty}$-norm of the resulting difference and applying Young's inequality gives
\begin{equation*}
\begin{split}
\| (&u^n - u)(t) \|_{L^{\infty}} \\
&\leq \| (u^n - u)(0) \|_{L^{\infty}}
	+ \| (a\Phi) * \nabla^{\perp}(\theta^n - \theta)(t) \|_{L^{\infty}}
	+ \| (a\Phi) * \nabla^{\perp}(\theta^n - \theta)(0) \|_{L^{\infty}}
\\
& + \int_0^t \| \nabla\nabla^{\perp}((1-a)\Phi) \|_{L^1} \left ( \| (\theta^n - \theta)(s)u^{n-1}(s) \|_{L^{\infty}} + \| \theta(s) (u^{n-1} - u)(s) \|_{L^{\infty}} \right)\, ds. 
 \end{split}
\end{equation*}   
It is clear that
\begin{equation*}
\begin{split}
& \| (u^n - u)(0) \|_{L^{\infty}} \rightarrow 0,\\
& \int_0^t  \left ( \| (\theta^n - \theta)(s)u^{n-1}(s) \|_{L^{\infty}} + \| \theta(s) (u^{n-1} - u)(s) \|_{L^{\infty}} \right)\, ds \rightarrow 0
\end{split}
\end{equation*}
as $n$ approaches infinity, for all $t\in[0,T]$.  

We now show that $\| (a\Phi) \ast \nabla^{\perp}(\theta^n - \theta)(t) \|_{L^{\infty}} = \| \nabla^{\perp}(a\Phi) \ast (\theta^n - \theta)(t) \|_{L^{\infty}}\rightarrow 0$ for all $t\in[0,T]$ as well.  We utilize that $\nabla^{\perp}(a\Phi)$ integrates to $0$.  Surpressing the time variable, and setting $\delta^n =  \theta^n - \theta$ for each $n$, we have, for any $\alpha\in(0,\min\{1,r-1\})$ and any $x\in\R^2$, 
\begin{equation*}
\begin{split}
&\abs{\nabla^{\perp}(a\Phi) * (\theta^n - \theta)(x)}
\leq \abs{\PV \int_{\R^2} \nabla^{\perp}(a\Phi) (y) \left( \delta^n (x-y) - \delta^n (x)\right) \, dy} +  \abs{C\delta_n(x)I}\\
&\leq \int_{\R^2} |\nabla (a\Phi) (y)| |y|^{\alpha} \left( \frac{|\delta^n(x-y) - \delta^n (x)|}{|y|^{\alpha}}\right) \, dy + \abs{C\delta_n(x)I}
\leq C \|  \delta^n \|_{{C}^{\alpha}}\rightarrow 0,
\end{split}
\end{equation*}
since $(\theta^n)$ converges to $\theta$ in $C([0,T]; C^{\alpha}(\R^2))$.
This implies (\ref{MainThmSerfatiSQG}) and completes the proof of the theorem. \\

\noindent \textbf{Uniqueness} An argument similar to the demonstration above that $(u^n)$ and $(\theta^n)$ are Cauchy gives uniqueness of solutions.
\end{proof}


\section{($SQG$) in uniformly local Sobolev spaces}\label{S:SQGSobolev}
\subsection{A priori estimates}\label{A Priori Estimates for Solutions in Uniformly Local Sobolev Spaces}

In this section, we establish a priori estimates on smooth solutions to ($SQG$) in uniformly local Sobolev spaces.  We prove the following theorem.
\begin{theorem}\label{mainQGEversion}
Assume $d=2$ and $s\geq 3$ is an integer.  Let $(u,\theta)$ be a solution to ($SQG$) on $[0,T]$ as given in Theorem \ref{SQGHolder} with \Holder exponent $r=s+2$.  Then 
\begin{equation*}
\| \theta(t) \|^2_{H^s_{ul}} \leq \| \theta^0 \|^2_{H^s_{ul}}\exp \left( C\int_0^t (\| u(\tau) \|_{\tilde{C}^1} + \|\nabla\theta(\tau)\|_{L^{\infty}}) \, d\tau \right).
\end{equation*}
\end{theorem}
\begin{proof}

Set $W=D^{\alpha}\theta$ with $0\leq |\alpha| \leq s$ and $s \geq  3$.  Apply $D^{\alpha}$ to $(\ref{SQG})_1$ to get 
\begin{equation}\label{vort1QGE}
\partial_t W + u\cdot\nabla W =   F,
\end{equation}
where 
\begin{equation*}
\begin{split}
& F = u\cdot\nabla W - D^{\alpha}(u\cdot\nabla \theta).\\
\end{split}
\end{equation*}  
Multiplying (\ref{vort1QGE}) by $\phi_{x}$ gives 
\begin{equation}\label{vort2QGE}
\begin{split}
&\partial_t (\phi_{x}W) + u\cdot\nabla (\phi_{x}W)   = (u\cdot\nabla\phi_{x})W+ \phi_{x}F .
\end{split}
\end{equation} 
After multiplying (\ref{vort2QGE}) by $\phi_{x}W$ and integrating, we conclude that
\begin{equation*}
\begin{split}
&\int_{\R^2} \phi_{x}W  \partial_t (\phi_{x}W)  + \int_{\R^2} \phi_{x}W  (u\cdot\nabla (\phi_{x}W))  \\
&\qquad   = \int_{\R^2} \phi_{x}W  (u\cdot\nabla\phi_{x})W+ \int_{\R^2} \phi_{x}W  \phi_{x} F .
\end{split}
\end{equation*}
Now observe that 
\begin{equation*}
\int_{\R^2} \phi_{x}W  \partial_t (\phi_{x}W) = \frac{1}{2} \frac{d}{dt} \| \phi_{x} W \|_{L^2}^2.
\end{equation*}
Moreover, one can show using the divergence-free property of $u$ and integration by parts that  
\begin{equation*}
\int_{\R^2} \phi_{x}W  (u\cdot\nabla (\phi_{x}W)) = 0.
\end{equation*}
By properties of $\phi_x$ and \Holders inequality, we also have
\begin{equation*}
\begin{split}
&\int_{\R^2} \phi_{x}W (u\cdot\nabla\phi_{x})W = \int_{\R^2} \phi_{x}W (u\cdot\nabla\phi_{x})\phi_{x,2}W \\
&\qquad \leq \| \phi_{x,2}W \|_{L^2}^2 \| u\cdot\nabla\phi_{x} \|_{L^{\infty}} \leq C \| u \|_{L^{\infty}}\| \theta \|^2_{H^s_{ul}}. 
\end{split}
\end{equation*}
Finally, another application of \Holders inequality gives 
\begin{equation*}
\int_{\R^2} \phi_{x}W \phi_{x} F \leq \| \phi_{x}W \|_{L^2} \| \phi_{x} F \|_{L^2}.
\end{equation*}
We apply Lemma \ref{commutator} to $\| \phi_{x} F \|_{L^2}$ with $f=u$ and $g=\nabla \theta$.  This gives
\begin{equation*}
\begin{split}
&\| \phi_{x} F \|_{L^2} \leq %
C\left( \| u \|_{\tilde{C}^1} \| \theta \|_{H^{s}_{ul}} + \| \nabla\theta \|_{L^{\infty}}\| u\|_{H^s_{ul}} \right). 
\end{split}
\end{equation*}
Combining the above estimates and integrating in time gives
\begin{equation*}
\begin{split}
&\| \phi_{x} W(t) \|_{L^2}^2  \leq C\| \phi_{x} W(0) \|_{L^2}^2\\
&\qquad + C\int_0^t \left( \| u \|_{L^{\infty}}\| \theta \|^2_{H^s_{ul}} + \| \theta \|_{H^s_{ul}}( \| u \|_{\tilde{C}^1} \| \theta \|_{H^{s}_{ul}} + \| \nabla\theta \|_{L^{\infty}}\| u\|_{H^s_{ul}} )\right) \, d\tau.
\end{split}
\end{equation*}
We now take the supremum of both sides of the inequality over $x\in\R^2$.  We conclude that
\begin{equation}\label{velvortboundQGE}
\begin{split}
&\| \theta(t) \|^2_{H^s_{ul}} \leq C\| \theta^0 \|^2_{H^s_{ul}}\\
&\qquad + C\int_0^t\left(\| u (\tau)\|_{\tilde{C}^1}\| \theta(\tau) \|^2_{H^{s}_{ul}} + \| \theta(\tau) \|_{H^{s}_{ul}}\|\nabla\theta (\tau)\|_{L^{\infty}}\| u(\tau) \|_{H^s_{ul}}  \right) \, d\tau. 
\end{split}
\end{equation}

It remains to close the estimate and apply \Gronwalls lemma.  To do this, note that, for each fixed $t\in[0,T]$, the approximating sequences $(\theta_n(t))$ and $(u_n(t))$ from the proof of Theorem \ref{SQGHolder} converge to $\theta(t)$ and $u(t)$, respectively, in $L^{\infty}(\R^2)$.   This convergence, along with Lemma \ref{L:CZLike}, allow us to pass to the limit in Lemma \ref{BMOSerfati}.  This gives, for all $j\in\Z$,
\begin{equation*}
\dot{\Delta}_j  u = \dot{\Delta}_j (\nabla^{\perp} (-\Delta)^{-1/2}\theta).
\end{equation*}
Applying a differential operator $D^{\gamma}$ with $1\leq |\gamma|\leq s$ gives 
\begin{equation*}
\dot{\Delta}_j D^{\gamma}  u = \dot{\Delta}_jD^{\gamma} (\nabla^{\perp} (-\Delta)^{-1/2}\theta).
\end{equation*}
This implies that, for $|\gamma|\geq 1$,
\begin{equation*}
\begin{split}
&D^{\gamma} u = D^{\gamma-1}P + (a\Phi) \ast \nabla^{\perp}D^{\gamma}\theta + [D^{\gamma}(\nabla^{\perp}((1-a)\Phi))] \ast \theta
\end{split}
\end{equation*}
for almost every $x\in\R^2$, where $P$ is a polynomial.  But $D^{\gamma}u$ and $D^{\gamma+1}\theta$ are in $C([0,T]; L^{\infty}(\R^2))$ for each $|\gamma|\leq s$, which implies that $P$ is a constant.  We conclude that for $2\leq |\gamma| \leq s$, $D^{\gamma}u$ and $D^{\gamma}\theta$ satisfy 
\begin{equation}\label{BSReplacementQGEGradu}
D^{\gamma} u = (a\Phi) \ast \nabla^{\perp}D^{\gamma}\theta + [D^{\gamma}(\nabla^{\perp}((1-a)\Phi))] \ast \theta
\end{equation} 
for almost every $x\in\R^2$.  Applying Lemma \ref{L:ConvHsBound}, for any multi-index $\beta$ with $|\beta|=2$,
\begin{equation}\label{SQGCZest}
\| D^{\beta} u \|_{H^{s-2}_{ul}} \leq C(\| \theta \|_{H^s_{ul}} + \| \theta \|_{L^{\infty}}) \leq C\| \theta \|_{H^s_{ul}}, 
\end{equation}
where we applied the Sobolev embedding theorem to get the last inequality.  This estimate, combined with $\| u \|_{L^2_{ul}} \leq C\| u \|_{L^{\infty}}$ and $\| \nabla u \|_{L^2_{ul}} \leq \| \nabla u \|_{L^{\infty}}$, gives
\begin{equation}\label{velocityHsulBound}
\| u \|_{H^{s}_{ul}}  \leq C(\| \theta \|_{H^s_{ul}} + \| u \|_{\tilde{C}^1} ).
\end{equation}
Substituting this estimate into (\ref{velvortboundQGE}) gives
\begin{equation}\label{preBKM}
\begin{split}
&\| \theta(t) \|^2_{H^s_{ul}} \leq \| \theta^0 \|^2_{H^s_{ul}}\\
& + C\int_0^t \left(\| u(\tau) \|_{\tilde{C}^1}\| \theta(\tau) \|^2_{H^{s}_{ul}} + \| \theta (\tau)\|_{H^s_{ul}}\|\nabla\theta(\tau)\|_{L^{\infty}}(\| \theta (\tau)\|_{H^s_{ul}} + \| u(\tau) \|_{\tilde{C}^1})  \right) \, d\tau\\
&\qquad \leq \| \theta^0 \|^2_{H^s_{ul}} + C\int_0^t (\| u(\tau) \|_{\tilde{C}^1} + \|\nabla\theta(\tau)\|_{L^{\infty}})\| \theta(\tau) \|^2_{H^{s}_{ul}}    \, d\tau\\
&\qquad + C\int_0^t  \| \theta (\tau)\|_{H^s_{ul}}\|\nabla\theta(\tau)\|_{L^{\infty}}\| u(\tau) \|_{\tilde{C}^1}    \, d\tau\\
&\qquad \leq \| \theta^0 \|^2_{H^s_{ul}} + C\int_0^t (\| u(\tau) \|_{\tilde{C}^1} + \|\nabla\theta(\tau)\|_{L^{\infty}})\| \theta(\tau) \|^2_{H^{s}_{ul}}    \, d\tau,
\end{split}
\end{equation}
where, to get the last inequality, we applied the Sobolev embedding theorem to conclude that $\|\nabla\theta(\tau)\|_{L^{\infty}} \leq C\| \theta(\tau)\|_{H^s_{ul}}$.  By \Gronwalls Lemma,
\begin{equation*}
\| \theta(t) \|^2_{H^s_{ul}} \leq \| \theta^0 \|^2_{H^s_{ul}}\exp \left( C\int_0^t (\| u(\tau) \|_{\tilde{C}^1} + \|\nabla\theta(\tau)\|_{L^{\infty}}) \, d\tau \right).
\end{equation*}  

This completes the proof of Theorem \ref{mainQGEversion}.
\end{proof}

\subsection{Existence of solutions}\label{S:Approx}

In this section, we prove the following theorem.
\begin{theorem}\label{SQGUnifLocalSobolev}
Let $s\geq 3$.  Let $\theta^0$ be a function in $H^s_{ul}(\R^2)$, and let $u^0$ in $H^s_{ul}(\R^2)$ satisfy
\begin{equation*}
u^0 =  \nabla^{\perp} (-\Delta)^{-1/2}  {\theta}^0 \qquad\text{  in  }\dot{C}^{\alpha}(\R^2),
\end{equation*}
where $\alpha>1$ satisfies the embedding $H^s_{ul}(\R^2)\hookrightarrow {C}^{\alpha}(\R^2)$.  There exists $T>0$ and a unique solution $(u, \theta)$ to 
\begin{equation}\label{SQG1}
\begin{split}
&\partial_t \theta + u\cdot\nabla\theta = 0,\\
&(u, \theta)|_{t=0} = (u^0, \theta^0)
\end{split}
\end{equation}   
satisfying 
\begin{equation*}
\begin{split}
&\theta \in L^{\infty}(0,T; H^s_{ul}(\R^2)) \cap Lip([0,T] ; H^{s-1}_{ul}(\R^2)),\\
&u \in L^{\infty}(0,T; H^s_{ul}(\R^2)). 
\end{split}
\end{equation*}     
Moreover, $(u,\theta)$ satisfies
\begin{equation}\label{hsulversion}
\begin{split}
&u(t) = u^0
+ ((a\Phi)) \ast \nabla^{\perp}(\theta(t) - \theta^0)
- \int_0^t (\nabla\nabla^{\perp}((1-a)\Phi)) \stardot (\theta u) \, ds.\\
\end{split}
\end{equation}  
\end{theorem}

\begin{proof}
For the proof of Theorem \ref{SQGUnifLocalSobolev}, we will construct an approximation sequence of smooth solutions $(u_n, \theta_n)$ given by Theorem \ref{SQGHolder} on $[0,T]$.  We will then use Theorem \ref{mainQGEversion} to establish uniform bounds on $(u_n, \theta_n)$ in the $H^s_{ul}$ norm, which will allow us to pass to the limit to obtain (\ref{SQG1}).\\
 \\
{\bf Approximation Sequence and Uniform Bounds.}  Consider the sequences $u^0_n = S_n u^0$ and $\theta^0_n=S_n \theta^0$. We see that for each $n$, $u^0_n$ and $\theta_n^0$ belong to $C^r(\R^2)$ for every $r>0$. Moreover, by Lemma \ref{smoothcutoffbounded}, there exists $\tilde{C}>0$, depending only on $s$, such that
 \begin{equation}\label{initialdatabounds}
 \begin{split}
 &\| u^0_n \|_{H^s_{ul}} \leq \tilde{C}\|u^0 \|_{H^s_{ul}},\\
 &\| \theta^0_n \|_{H^s_{ul}} \leq \tilde{C}\|\theta^0 \|_{H^s_{ul}}.
 \end{split}
 \end{equation}    
We claim there exists a single $T>0$ such that both $(u_n)$ and $(\theta_n)$ are uniformly bounded in $L^\iny(0, T; H^s_{ul}(\R^2))$.  To see that such a $T$ exists, note that Lemma \ref{SobolevEmbedding} and  construction of $u^0_n$ and $\theta^0_n$ give an $\alpha > 1$ such that both $(u^0_n)$ and $(\theta_n^0)$ are uniformly bounded in $C^{\alpha}(\R^2)$.  Thus, by Theorem \ref{SQGHolder}, a solution $(u_n,\theta_n)$ exists in $C^{\alpha}(\R^3)$ at least on $[0,T_n]$, with $(u_n,\theta_n)$ satisfying (\ref{shorttimebound}).  Choose $T>0$ such that, for every $n$, $T\leq T_n$ and 
$$ \frac{1}{2C} \leq T_n(\| u_n^0 \|_{L^{\infty}} + \| \theta^0_n \|_{C^{\alpha}}) \leq T(\| u^0 \|_{L^{\infty}} + \| \theta^0 \|_{C^{\alpha}}) < \frac{1}{C},$$
where $C$ is as in (\ref{shorttimebound}).  We have that for every $n$, $(u_n,\theta_n)$ is a solution satisfying Theorem \ref{SQGHolder} on $[0,T]$.  In particular, by (\ref{shorttimebound}), 
\begin{equation}\label{Holderuniform}
\| u_n \|_{C([0,T]; L^{\infty})} + \| \theta_n \|_{C([0,T]; C^{\alpha})} \leq \frac{C(\| u^0 \|_{L^{\infty}} + \| \theta^0 \|_{C^{\alpha}})}{1- CT(\| u^0 \|_{L^{\infty}} + \| \theta^0 \|_{C^{\alpha}})},
\end{equation}
and $\| u_n \|_{C([0,T]; L^{\infty})}$, $\| \theta_n \|_{C([0,T]; C^{\alpha})}$ are therefore uniformly bounded in $n$.  

To establish a uniform bound on $\| u_n \|_{C([0,T]; C^{\alpha})}$ in $n$, note that by Lemma \ref{holderrieszbound},
 \begin{equation}\label{velocityHolderbound}
 \| u_n \|_{C([0,T]; C^{\alpha})} \leq C(\| u_n \|_{C([0,T]; L^{\infty})} + \| \theta_n \|_{C([0,T]; C^{\alpha})}).
 \end{equation}  
 Then the uniform bound on $\| u_n \|_{C([0,T]; C^{\alpha})}$ again follows from (\ref{Holderuniform}).

Theorem \ref{mainQGEversion}, (\ref{initialdatabounds}), and an application of the uniform bound on $\| u_n \|_{C([0,T]; C^{\alpha})}$ imply that there exists a constant $C>0$, depending only on the initial data and $T$, such that
\begin{equation}\label{vortunifhsulbound}
\| \theta_n \|_{C([0,T]; H^s_{ul})} \leq C.
\end{equation}    
This bound, combined with the estimate (\ref{velocityHsulBound}), imply that there exists a constant $C>0$, depending only on the initial data and $T$, such that
\begin{equation}\label{velunifhsulbound}
\| u_n \|_{C([0,T]; H^s_{ul})} \leq C
\end{equation}    
as well.

To simplify notation in what follows, we set $\phi_R = \phi_{0,R}$, where $\phi_{0,R}$ is as in (\ref{bumpnotation}).  \\

\noindent {\bf ($\bm{\phi_R\theta_n}$) is Cauchy.}  We now show that ($\phi_R\theta_n$) is a Cauchy sequence in the space $C([0,T];H^{s-1}(\R^2))$ for every $R>0$.  For some $\alpha>1$ and for each $n$, we know that $u_n$, $\theta_n\in C([0,T];C^{\alpha}(\R^2))$, and that our solutions satisfy
    \begin{equation}\label{SQGeqseq}
    \partial_t \theta_n + u_n\cdot\nabla \theta_n=0.
    \end{equation} 
Multiplying (\ref{SQGeqseq}) by $\phi_R$ for some fixed $R>0$, we have
\begin{equation}\label{uniftimederbound}
\begin{split}
&\|\phi_R \partial_t \theta_n\|_{H^{s-1}} \leq \|\phi_R u_n \cdot \nabla \theta_n\|_{H^{s-1}}\leq C(R)\|u_n \cdot \nabla \theta_n\|_{H^{s-1}_{ul}}\\
&\qquad \leq C(R)\|u_n\|_{H^{s-1}_{ul}}\|\nabla \theta_n\|_{H^{s-1}_{ul}} \leq C(R),
\end{split}
\end{equation}
where the third inequality follows because $H^{s-1}_{ul}(\R^2)$ is a Banach algebra, and the last inequality follows since $\| \nabla \theta_n \|_{H^{s-1}_{ul}}$ and $\|u_n\|_{H^{s-1}_{ul}}$ are uniformly bounded in $n$ by a quantity depending only on the initial data.
    
Via Rellich's theorem, since for each $t\in[0,T]$, $\|\theta_n(t)\|_{H^s_{ul}}$ is uniformly bounded over $n$, we can conclude that for each $R$ and $t \in [0,T]$, there exists a subsequence of $(\phi_R \theta_n(t))$ which converges in $H^{s-1}(\R^2)$. A diagonalization argument shows that for each $t\in[0,T],$ there is a subsequence of $(\theta_n(t))$ (relabeled to $(\theta_n(t))$) such that for every $R>0$, the sequence $(\phi_R \theta_n(t))$ converges in $H^{s-1}(\R^2)$. 

It remains to find a subsequence which converges for all $t\in [0,T].$ From (\ref{uniftimederbound}), it follows that given $\epsilon>0$, there exists $\delta>0$ such that for all $n$ and for all $s,t\in[0,T]$ such that $|t-s|<\delta$,
 \begin{equation}\label{omegacont}   
 \|\phi_R \theta_n(t)-\phi_R \theta_n(s)\|_{H^{s-1}}<\epsilon/3.
 \end{equation}
    Consider a partition of $[0,T]$, $0=t_0<t_1<...<t_M=T$ such that $t_i-t_{i-1}<\delta$. Since there are finitely many elements in the partition, we can find a further subsequence of $(\phi_R \theta_n)$ (relabeled as $(\phi_R \theta_n)$) such that for each $t_i$ in our partition, $(\phi_R \theta_n(t_i))$ converges in $H^{s-1}(\R^2)$ for all $R>0$. Let $N$ be such that for all $m,n\geq N$ and for all $t_i$ in our partition, 
    $$\|\phi_R\theta_n(t_i)-\phi_R\theta_m(t_i)\|_{H^{s-1}}<\epsilon/3.$$
    It follows that for all pairs $m,n\geq N$ and for each $t\in[0,T]$, with $t_i$ chosen to satisfy $|t-t_i|<\delta$,
    
\begin{equation*}
\begin{split}
&\|  \phi_R  \theta_n(t) -  \phi_R  \theta_m(t) \|_{H^{s-1}} \leq \| \phi_R \theta_n(t) - \phi_R \theta_n(t_i) \|_{H^{s-1}}\\
&\qquad   + \| \phi_R \theta_n(t_i) - \phi_R \theta_m(t_i) \|_{H^{s-1}}  + \| \phi_R \theta_m(t_i) - \phi_R \theta_m(t) \|_{H^{s-1}} < \epsilon. 
\end{split}
\end{equation*}
Therefore, ($\phi_R \theta_n$) is a Cauchy sequence in $C([0,T];H^{s-1}(\R^2))$, and  thus converges in $C([0,T];H^{s-1}(\R^2))$.
We conclude that there exists $\theta$ such that $\phi_R \theta_n \to \phi_R \theta$ in $C([0, T]; H^{s - 1}(\R^2))$ for all $R > 0$.
 \\

\noindent {\bf $\bm{(\phi_Ru_n)}$ is Cauchy.}  The proof that for all $R>0$, $(\phi_Ru_n)$ is also Cauchy in $C([0,T];H^{s-2}(\R^2))$ is similar.  Indeed, for each $t\in[0,T]$, the uniform bound on $\|u_n(t)\|_{H^s_{ul}}$ over $n$ and a diagonalization argument, as above, allow us to conclude that there exists a subsequence of $(\phi_R u_n(t))$ which converges in $H^{s-2}(\R^2)$ for every $R>0$. It remains to find a single subsequence which converges for all $t\in [0,T]$.  We observe that by Theorem \ref{SQGHolder}, for $s,t\in[0,T]$, 
\begin{equation*}
u_n(t) - u_n(s) =   (a \Phi) * \grad^\perp(\theta_n(t) - \theta_n(s))- \int_s^t (\grad \grad^\perp ((1 - a) \Phi)) \stardot (\theta_n u_n),
\end{equation*}  
so that %
for each $R>0$,
\begin{equation*}
\begin{split}
&\| \phi_Ru_n(t) - \phi_Ru_n(s)\|_{H^{s-2}} \leq C(R)\| \phi_{8R}\theta_n(t) - \phi_{8R}\theta_n(s) \|_{H^{s-1}}+ C(R)\int_s^t  \|\theta_n\|_{L^{\infty}} \|u_n\|_{L^{\infty}} \\
&\qquad \leq C(R)|t-s| + C(R)|t-s|\sup_{\tau\in[s,t]} \|\theta_n(\tau)\|_{L^{\infty}} \|u_n(\tau)\|_{L^{\infty}},
\end{split}
\end{equation*}  
where we used the equality $\phi_R((a\Phi)\ast f)= \phi_R((a\Phi)\ast (\phi_{8R}f))$ to get the first inequality, and we used (\ref{uniftimederbound}) to get the second inequality.  Since $\|\theta_n\|_{L^{\infty}}$ and $\|u_n\|_{L^{\infty}}$ are uniformly bounded in $n$, given $\epsilon>0$, there exists $\delta>0$ such that for all $R>0$, whenever $s,t\in[0,T]$ satisfy $|s-t|<\delta$,
$$ \| \phi_Ru_n(t) - \phi_Ru_n(s)\|_{H^{s-2}} < \epsilon. $$ 
Following an argument identical to that used to show ($\phi_R \theta_n$) is a Cauchy sequence in $C([0,T];H^{s-1}(\R^2))$, we can conclude that ($\phi_R u_n$) is Cauchy in $C([0,T];H^{s-2}(\R^2))$, and there exists $u$ with $\phi_R u_n \to \phi_R u$ in $C([0, T]; H^{s - 2}(\R^2))$ for all $R > 0$.
  \\

\noindent{\bf ($\bm{u,\theta}$) satisfies Theorem \ref{mainQGEversion}.}  We now pass to the limit in the $H^{s-2}(\R^2)$ norm. Given $R>0$, if we multiply (\ref{SQGeqseq}) by $\phi_R$, then for $n,m\in \N,$

\begin{equation*}
\begin{split}
&\phi_R(\partial_t \theta_n - \partial_t \theta_m) = \phi_R(u_n - u_m)\cdot\nabla \theta_m + \phi_R u_n\cdot\nabla(\theta_n - \theta_m)\\
&= \phi_R(u_n - u_m)\cdot(\phi_{2R}\nabla \theta_m) + \phi_R u_n\cdot\phi_{2R}\nabla(\theta_n - \theta_m).
\end{split}
\end{equation*}  
    Hence, at each $t\in[0,T]$,  
    \begin{equation*}
    \begin{split}
        &\|\phi_R (\partial_t \theta_n- \partial_t \theta_m)\|_{H^{s-2}} \leq \|\phi_R(u_n - u_m)\cdot(\phi_{2R}\nabla \theta_m)\|_{H^{s-2}}\\
        &\qquad + \|\phi_R u_n\cdot\phi_{2R}\nabla(\theta_n - \theta_m)\|_{H^{s-2}}\\
        &\leq \|\phi_R(u_n - u_m)\|_{H^{s-2}}\|\phi_{2R}\nabla\theta_m\|_\infty + \|\phi_R(u_n - u_m)\|_\infty \|\phi_{2R}\nabla\theta_m\|_{H^{s-2}}\\
        & + \|\phi_R u_n\|_{H^{s-2}}\|\phi_{2R}\nabla(\theta_n - \theta_m)\|_\infty +  \|\phi_R u_n\|_\infty \|\phi_{2R}\nabla(\theta_n - \theta_m)\|_{H^{s-2}}.
    \end{split}
    \end{equation*}
    Since $\|\phi_{2R}\nabla\theta_m\|_\infty$, $\|\phi_{2R}\nabla \theta_m\|_{H^{s-2}}$, $\|\phi_R u_n\|_{H^{s-2}}$, and  $\|\phi_R u_n\|_\infty$ are uniformly bounded in $n$ on $[0,T]$, as $N \to \iny$, we have
\begin{equation*} 
\begin{split}
& \sup_{m, n \ge N} \|\phi_R(u_n - u_m)\|_{H^{s-2}}\|\phi_{2R}\nabla\theta_m\|_\infty \rightarrow 0,\\
& \sup_{m, n \ge N} \|\phi_R(u_n - u_m)\|_\infty \|\phi_{2R}\nabla\theta_m\|_{H^{s-2}} \rightarrow 0,\\
& \sup_{m, n \ge N} \|\phi_R u_n\|_{H^{s-2}}\|\phi_{2R}\nabla(\theta_n - \theta_m)\|_\infty \rightarrow 0, \\
&  \sup_{m, n \ge N} \|\phi_R u_n\|_\infty \|\phi_{2R}\nabla(\theta_n - \theta_m)\|_{H^{s-2}} \rightarrow 0.
\end{split}
\end{equation*}
From these estimates, it follows that $(\phi_R \partial_t \theta_n)$ is Cauchy in $C([0,T];H^{s-2}(\R^2)).$

Since $\phi_R\theta_n \to \phi_R \theta$ in $C([0,T]\times \R^2)$, we also have $\phi_R \theta_n \to \phi_R \theta$ in $\mathcal{D}'([0,T]\times \R^2)$. This implies that $\phi_R\partial_t\theta_n\to\phi_R\partial_t\theta$ in $\mathcal{D}'([0,T]\times \R^2)$, so by the uniqueness of limits, for all $R>0,$ $\phi_R\partial_t\theta_n\to\phi_R\partial_t\theta$ in $C([0,T]; H^{s-2}(\R^2))$. 

We multiply $(\ref{SQG1})_1$, as satisfied by ($u_n,\theta_n$), by $\phi_R$, and we pass to the limit in $C([0,T]; H^{s-2}(\R^2))$.  This gives $\phi_R\partial_t \theta=-\phi_R u\cdot \nabla \theta$ for all $R>0$.  

To see that $\theta$ belongs to $L^{\infty}(0,T; H^s_{ul}(\R^2))$, we use (\ref{vortunifhsulbound}) and a weak-* compactness argument.  Note that by (\ref{vortunifhsulbound}), for every $x\in\R^2$, $n\in\N$, and $t\in[0,T]$, $$\| \phi_x \theta_n(t) \|_{H^s} \leq C.$$  Therefore, up to a subsequence which depends on $t$ and $x$, $\phi_x \theta_n(t)$ converges weak-* in $H^s(\R^2)$.  Note, however, that for every $R>0$ and $t\in[0,T]$, $\phi_R\theta_n(t)\rightarrow \phi_R\theta(t)$ in $H^{s-1}(\R^2)$.  Given $x$, since we can always choose $R$ large enough to ensure that $\phi_x = \phi_x\phi_R$, we have $\phi_x\theta_n(t)\rightarrow \phi_x\theta(t)$ in $L^2(\R^2)$.  By uniqueness of limits, $\phi_x \theta_n(t)$ converges weak-* in $H^s(\R^2)$ to $\phi_x \theta(t)$, and 
$$ \| \phi_x \theta(t) \|_{H^s} \leq C.$$
This holds for all $t\in [0,T]$ and for all $x\in\R^2$, so $\theta$ belongs to $L^{\infty}(0,T; H^s_{ul}(\R^2))$.  The argument showing that $u$ belongs to $L^{\infty}(0,T; H^s_{ul}(\R^2))$ is similar.\\
\\
{\bf ($\bm{u,\theta}$) satisfies (\ref{hsulversion}).}  Note that (\ref{hsulversion}) follows from Theorem \ref{SQGHolder} since, by the Sobolev Embedding Theorem, $u$ and $\theta$ belong to $C([0,T];C^{\alpha}(\R^2))$ for some $\alpha>1$. 
This completes the proof of Theorem \ref{SQGUnifLocalSobolev}.  %
\\

\noindent \textbf{Uniqueness} Applying a cutoff function $\phi_R$ to two solutions and making the same argument that showed $(\phi_R u_n)$ is Cauchy gives uniqueness of solutions.
\end{proof}
 
\section{($E$) in uniformly local Sobolev spaces}\label{S:Euler3D}
\subsection{A priori estimates}
We now prove an analogous theorem to Theorem \ref{mainQGEversion} for the Euler equations.
\begin{theorem}\label{mainLpversion}
Assume $s$ is an integer satisfying $s \geq 3$, with $d=2$ or $3$.  Let $u$ be a solution to ($E$) in $C^1([0,T]; H^k(\R^d))$ for all $k\in\N$.  Then there exists $C>0$, depending on $s$, such that the following estimate holds for each $t\in[0,T]$:
\begin{equation}\label{umainest}
 \| \omega (t)\|^2_{H^{s-1}_{ul}} \leq (1+ \| \omega^0\|^2_{H^{s-1}_{ul}})\exp\left( C\int_0^t \| u (\tau)\|_{\tilde{C}^1}(\| u(\tau) \|^2_{L^{\infty}}  + 1) \, d\tau \right).
\end{equation}
\end{theorem}
\begin{proof}
The proof is similar to that of Theorem \ref{mainQGEversion}.  We prove the theorem for $d=3$.  The proof clearly extends to the case $d=2$.

Set $W=D^{\alpha}\omega$ with $0\leq |\alpha| \leq s-1$ and $s\geq 3$.  Apply $D^{\alpha}$ to the vorticity equation to get 
\begin{equation}\label{vort1}
\partial_t W + u\cdot\nabla W = D^{\alpha}(\omega\cdot \nabla u)  +F,
\end{equation}
where 
\begin{equation*}
\begin{split}
& F = u\cdot\nabla W - D^{\alpha}(u\cdot\nabla \omega).\\
\end{split}
\end{equation*}  
For $x\in\R^3$ fixed, multiply (\ref{vort1}) by $\phi_{x}$ to get 
\begin{equation}\label{vort2}
\begin{split}
&\partial_t (\phi_{x}W) + u\cdot\nabla (\phi_{x}W)  = (u\cdot\nabla \phi_{x})W+ \phi_{x}D^{\alpha}(\omega\cdot \nabla u) + \phi_{x}F .
\end{split}
\end{equation} 
After taking the dot product of (\ref{vort2}) with $\phi_{x}W$ and integrating, we conclude that
\begin{equation*}
\begin{split}
&\int_{\R^3} \phi_{x}W \cdot \partial_t (\phi_xW)  + \int_{\R^3} \phi_{x}W \cdot (u\cdot\nabla (\phi_{x}W))   = \int_{\R^3} \phi_{x}W \cdot ((u\cdot\nabla\phi_{x})W)\\
&\qquad+ \int_{\R^3} \phi_{x}W \cdot (\phi_{x}D^{\alpha}(\omega\cdot \nabla u)) + \int_{\R^3} \phi_{x}W \cdot \phi_{x} F .
\end{split}
\end{equation*}
Now observe that 
\begin{equation*}
\int_{\R^3} \phi_{x}W \cdot \partial_t (\phi_x W) = \frac{1}{2} \frac{d}{dt} \| \phi_{x} W \|_{L^2}^2.
\end{equation*}
Moreover, one can show using the divergence-free property of $u$ and integration by parts that  
\begin{equation*}
\int_{\R^3} \phi_{x}W \cdot (u\cdot\nabla (\phi_{x}W)) = 0.
\end{equation*}
By properties of $\phi_x$ and \Holders inequality, we also have
\begin{equation*}
\begin{split}
&\int_{\R^3} \phi_{x}W \cdot ((u\cdot\nabla\phi_{x})W) = \int_{\R^3}\phi_{x}W \cdot ((u\cdot\nabla\phi_{x})\phi_{x,2}W) \\
&\qquad \leq \| \phi_{x,2}W \|_{L^2}^2 \| u\cdot\nabla\phi_{x} \|_{L^{\infty}} \leq C \| u \|_{L^{\infty}}\| \omega \|^2_{H^{s-1}_{ul}}, 
\end{split}
\end{equation*}
and, again from \Holders inequality,
\begin{equation*}
\begin{split}
&\int_{\R^3} \phi_{x}W \cdot (\phi_{x}D^{\alpha}(\omega\cdot \nabla u))  \leq C\| \phi_{x}W \|_{L^2}  \| \phi_{x}D^{\alpha}(\omega\cdot \nabla u)\|_{L^2}\\
&  \leq \| \omega \|_{H^{s-1}_{ul}} (\|\omega \|_{H^{s-1}_{ul}}\| \nabla  u \|_{L^{\infty}} + \|\omega \|_{L^{\infty}}\| \nabla  u \|_{H^{s-1}_{ul}}) \leq C \| u \|^{2}_{H^{s}_{ul}}\| \nabla  u \|_{L^{\infty}},
\end{split}
\end{equation*}
where we used Lemma \ref{commutator} to get the second inequality.  Finally, another application of \Holders inequality gives 
\begin{equation*}
\int_{\R^3} \phi_{x}W \cdot \phi_{x} F \leq \|\phi_{x} W \|_{L^2}  \| \phi_{x} F \|_{L^2}.
\end{equation*}
Since $u$ is divergence free, we can write 
\begin{equation*}
F = u\cdot\nabla W - D^{\alpha}\dv(u \omega),
\end{equation*} 
which allows us to apply Lemma \ref{commutator} to $\| \phi_{x} F \|_{L^2}$ with $f=u$ and $g=\omega$.  This gives
\begin{equation*}
\| \phi_{x} F \|_{L^2} \leq C\left( \| u \|_{\tilde{C}^1} \| \omega \|_{H^{s-1}_{ul}} + \| \omega \|_{L^{\infty}}\| u\|_{H^{s}_{ul}} \right) \leq C\| u \|_{\tilde{C}^1}\| u\|_{H^{s}_{ul}}. 
\end{equation*}
Combining the above estimates gives
\begin{equation*}
\frac{1}{2} \frac{d}{dt} \| \phi_{x} W \|_{L^2}^2 \leq C\| u \|_{\tilde{C}^1}\left( \| \omega \|_{H^{s-1}_{ul}}\| u\|_{H^{s}_{ul}} + \| u\|^2_{H^{s}_{ul}} \right) \leq C\| u \|_{\tilde{C}^1} \| u\|^2_{H^{s}_{ul}}.
\end{equation*}
After integrating in time and taking the supremum over $x\in\R^3$ of both sides, we conclude that 
\begin{equation}\label{velvortbound}
\begin{split}
&\| \omega(t) \|^2_{H^{s-1}_{ul}} \leq \| \omega^0 \|^2_{H^{s-1}_{ul}} + C\int_0^t \| u (\tau)\|_{\tilde{C}^1} \| u(\tau)\|^2_{H^{s}_{ul}} \, d\tau. 
\end{split}
\end{equation}
It remains to close the estimate and apply \Gronwalls lemma.  To do this, we use the Biot-Savart law. %

Let $K_3 = \grad G$, with $G$ as in (\ref{e:G}), be (one form of) the Biot-Savart kernel in dimension $3$.  Setting $\omega^i_k= (\grad u - (\grad u)^T)^i_k = \partial_k u^i - \partial_i u^k$, since $u$ and $\omega$ are smooth and decaying, for $1\leq i \leq 3$, using implicit sum notation,
\begin{equation}\label{BSReplacementEuler}
\begin{split}
&u^i = K_3^k \ast \omega^i_k =  (aK_3^k) \ast \omega^i_k + ((1-a)K_3^k) \ast \omega^i_k \\
&\qquad = (aK_3^k) \ast \omega^i_k + (\partial_k((1-a)K_3^k)) \ast u^i - (\partial_i((1-a)K_3^k)) \ast u^k.
\end{split}
\end{equation}
Applying a differential operator $D^{\beta}$, with $0\leq |\beta| \leq s-1$, to both sides of (\ref{BSReplacementEuler}) gives
\begin{equation*} %
	D^{\beta} u^i
		=  (aK_3^k) \ast D^{\beta}\omega^i_k
			+ [D^{\beta}\partial_k((1-a)K_3^k)] \ast u^i
			- [D^{\beta}\partial_i((1-a)K_3^k)] \ast u^k.
\end{equation*} 
Setting $D^{\gamma} = \partial_jD^{\beta}$ and applying $\partial_j$
then gives
\begin{equation}\label{BSReplacementEulerGradu}
\begin{split}
D^{\gamma} u^i
	&= \partial_j((aK_3^k) \ast D^{\beta}\omega^i_k)
		+ [D^{\gamma} \prt_i ((1-a)K_3^k)] \ast u^i.
\end{split}
\end{equation} 
Applying Lemma \ref{L:ConvHsBound} gives
\begin{equation*}
\| \nabla u \|_{H^{s-1}_{ul}} \leq C(\| \omega \|_{H^{s-1}_{ul}} + \| u \|_{L^{\infty}}).
\end{equation*}
This estimate, combined with $\| u \|_{L^2_{ul}} \leq C\| u \|_{L^{\infty}}$, gives
\begin{equation}\label{uomegasobolevbound}
\| u \|_{H^{s}_{ul}}  \leq C(\| \omega \|_{H^{s-1}_{ul}} + \| u \|_{L^{\infty}} ).
\end{equation}
We use (\ref{uomegasobolevbound}) and (\ref{velvortbound}) to write
\begin{equation*}
\begin{split}
&\| \omega (t)\|^2_{H^{s-1}_{ul}} \leq  \| \omega^0 \|^2_{H^{s-1}_{ul}} + C\int_0^t \| u (\tau)\|_{\tilde{C}^1}(\| \omega(\tau) \|_{H^{s-1}_{ul}} + \| u(\tau) \|_{L^{\infty}} )^2  \, d\tau \\
&\qquad \leq  \| \omega^0 \|^2_{H^{s-1}_{ul}} + C\int_0^t \| u (\tau)\|_{\tilde{C}^1}(\| \omega(\tau) \|^2_{H^{s-1}_{ul}} + \| u(\tau) \|^2_{L^{\infty}} )  \, d\tau,\\
&\qquad \leq  \| \omega^0 \|^2_{H^{s-1}_{ul}} + C\int_0^t \| u (\tau)\|_{\tilde{C}^1}(\| u(\tau) \|^2_{L^{\infty}}  + 1)(\| \omega(\tau) \|^2_{H^{s-1}_{ul}} + 1 )  \, d\tau,
\end{split}
\end{equation*}
where we used that for $A,B\geq 0$, $(A+B)^2 \leq C(A^2 + B^2)$ to get the second inequality.  Setting $h(t) = 1 + \| \omega (t)\|^2_{H^{s-1}_{ul}}$, we have 
\begin{equation*}
h(t)  \leq h(0) + C\int_0^t \| u (\tau)\|_{\tilde{C}^1}(\| u(\tau) \|^2_{L^{\infty}}  + 1)h(\tau)  \, d\tau.
\end{equation*}  
An application of \Gronwalls Lemma gives
\begin{equation*}
 \| \omega (t)\|^2_{H^{s-1}_{ul}} \leq (1+ \| \omega^0\|^2_{H^{s-1}_{ul}})\exp\left( C\int_0^t \| u (\tau)\|_{\tilde{C}^1}(\| u(\tau) \|^2_{L^{\infty}}  + 1) \, d\tau \right).
\end{equation*}
This completes the proof of Theorem \ref{mainLpversion}.
\end{proof}

\subsection{Existence of solutions} 

We prove the following theorem.

\begin{theorem}\label{EulerUnifLocalSobolev}
Let $s\geq 3$.  Let $u^0$ be a function in $H^s_{ul}(\R^3)$, and let $\omega^0=\nabla\times u^0$.  There exists $T>0$ and a unique classical solution $(u, p)$ to (E) satisfying 
\begin{equation*}
\begin{split}
&u \in L^{\infty}(0,T; H^s_{ul}(\R^3)) \cap Lip([0,T] ; H^{s-1}_{ul}(\R^3)).
\end{split}
\end{equation*}    

Moreover, $p$ satisfies  %
    \begin{align}\label{e:gradpExp}
    \begin{split}
        \grad p(x)
			&= - \int_{\R^3} a(x - y) \grad G (x - y)
				        \dv \dv (u \otimes u)(y) \, dy \\
			    &\qquad
	 			    + \int_{\R^3} (u \otimes u)(y) \cdot
					    \grad \grad \brac{(1 - a(x - y)) \grad G(x - y)}
					    \, dy.
	    \end{split}
    \end{align} 
\end{theorem}

To prove the theorem, we construct an approximation sequence of smooth, decaying solutions to ($E$), and we pass to the limit in ($E$).  The construction of the sequence of initial velocities is slightly more tedious in the three-dimensional setting than in two dimensions, as we must make use of a more complicated explicit formula for a three-dimensional stream function.

Because we are seeking a strong solution to ($E$) in $H^s_{ul}$, we are forced to consider the meaning of the pressure for such solutions.  We are able to make sense of the pressure by passing to a certain limit of the sequence of smooth pressures generated from the smooth velocity solutions.       

To prepare the initial velocity, we adapt the classical strategy employed in \cite{AKLN} and \cite{Cozzi} of cutting off and smoothing the stream function associated with the initial velocity $u^0$. Some additional care is required because of the lack of inherent decay of the velocity field.

\begin{lemma}\label{unifbound}
Let $u^0\in H^s_{ul}(\R^3)$ and let $s$ be a nonnegative integer. There exists a sequence $(u^0_n)$ of Schwarz-class, divergence-free vector fields uniformly bounded in $H^s_{ul}(\R^3)$ for which $\phi_R u_n^0 \to \phi_R u^0$ in ${H^s(\R^3)}$ for any fixed $R > 0$.
\end{lemma} 
\begin{proof}
Let $(m_n)_{n = 1}^\iny$ be a sequence of positive integers that we will specify later. We define $u^0_n$ by
\begin{align}\label{curlprodrule}
	u^0_n &=
		S_{m_n} (\nabla\times (\phi_n \psi))
		= S_{m_n} (\phi_n u^0) + S_{m_n} (\nabla \phi_n \times \psi),
\end{align}
where $\psi$ is the stream function for $u^0$ given in Lemma \ref{L:Stream}.
Observe that $u^0_n$ is Schwarz-class and divergence-free.

Using Lemmas \ref{commutator} and \ref{smoothcutoffbounded},
$$
	\|S_{m_n}(\phi_n u^0)\|_{H^s_{ul}}
		\leq C\| \phi_n u^0\|_{H^s_{ul}}
		\leq C\| \phi_n \|_{C^s} \| u^0\|_{H^s_{ul}}
		\leq C\| u^0\|_{H^s_{ul}}.
$$
Again by Lemmas \ref{commutator} and \ref{smoothcutoffbounded},
\begin{align*}
	&\|S_{m_n} (\nabla \phi_n\times \psi)\|_{H^s_{ul}}
		\le C \|\nabla \phi_n\times \psi\|_{H^s_{ul}}
		= C \sup_{z \in \R^3} \|\phi_z \nabla \phi_n\times \psi\|_{H^s} \\
		&\qquad
		= C \sup_{z \in \R^3} \|\phi_z \nabla \phi_n\times \psi\|_{H^s(U)}
		\le C \sup_{z \in \R^3}
			\|\phi_z^{\frac{1}{2}} \nabla \phi_n\|_{\tilde{C}^s(U)}
			\|\phi_z^{\frac{1}{2}} \psi\|_{H^s(U)},
\end{align*}
where $U = B_{2n}(0) \cap B_2(z)$.
But by Lemma \ref{L:Stream},
$
	\norm{\psi}_{H^s(U)}
		\le C n \smallnorm{u^0}_{H^s_{ul}}
		= Cn
$
and hence
$\|\phi_z^{\frac{1}{2}} \psi\|_{H^s(U)} \le C n$.
Since $\grad \phi_n(\cdot) = n^{-1} \grad \phi(n^{-1} \cdot)$, we have $\|\phi_z^{\frac{1}{2}} \nabla \phi_n\|_{\tilde{C}^s(U)} \le C n^{-1}$. It follows that
\begin{align*}
	\norm{S_{m_n} (\nabla \phi_n \times \psi)}_{H^s_{ul}}
		\le \frac{C}{n} C n
		= C.
\end{align*}

This shows that $(u^0_n)$ is uniformly bounded in $H^s_{ul}(\R^3)$.

We now show that $\phi_R u_n^0 \to \phi_R u^0$ in ${H^s(\R^3)}$. Because $S_{m_n}$ commutes with $\nabla \times$,
\begin{align*}
	\phi_R (u^0_n - u^0)
		&= \phi_R \brac{S_{m_n} (\nabla \times (\phi_n \psi)) - \nabla \times \psi}
		= \phi_R \nabla \times (S_{m_n} (\phi_n \psi) - \psi).
\end{align*}
But,
\begin{align*}
	S_{m_n} (\phi_n \psi) - \psi
		&= S_{m_n} (\phi_n \psi) - \phi_n \psi + (\phi_n - 1) \psi.
\end{align*}
For $n > 2R$, $\phi_n - 1  = 0$, so $\phi_R \nabla \times ((\phi_n - 1) \psi) = 0$, leaving
\begin{align*}
	\phi_R (u^0_n - u^0)
		&= \phi_R \nabla \times \pr{S_{m_n} (\phi_n \psi) - \phi_n \psi}.
\end{align*}

It is now time to choose $m_n$. Because $\nabla\times (\phi_n \psi) \in H^{s}(\R^3)$, we know that
$
	\nabla \times(S_k (\phi_n \psi) - \phi_n \psi)
		\to 0 \text{ in } H^{s}(\R^3),
$
as $k \to \iny$, so choose $m_n \ge n$ sufficiently large that 
\begin{align*}
	\norm{\nabla \times(S_k (\phi_n \psi) - \phi_n \psi)}_{H^{s}(\R^3)} \le \frac{1}{n}
		\text{ for all } k \ge m_n.
\end{align*}
It follows that
\begin{align*}
	\norm{\phi_R (u^0_n - u^0)}_{H^s(\R^3)}
		&\le \frac{C(R)}{n}.
\end{align*}
This gives $\phi_R u_n^0 \to \phi_R u^0$ in ${H^s(\R^3)}$ for any $R > 0$.
\end{proof}

\begin{proof}[\textbf{Proof of Theorem \ref{EulerUnifLocalSobolev}}] $\phantom{x}$

\noindent\textbf{Uniform time of existence.}
From the sequence of initial velocities in Lemma \ref{unifbound}, we generate a sequence $( u_n)$ of solutions to ($E$) in $H^k(\R^3)$ for all $k$, where the time interval of existence in $H^k(\R^3)$ for each $u_n$ may vary with $n$. We claim, however, that there exists a single $T>0$ such that $u_n$ solves ($E$) with $(u_n)$ uniformly bounded in $L^\iny(0, T; H^s_{ul}(\R^3))$.

 To see that such a $T$ exists, note that Lemma \ref{SobolevEmbedding} gives an $\alpha > 1$ such that $u^0_n$ belongs to $C^{\alpha}(\R^3)$ for each $n$.  Thus, a solution $u_n$ will exist in $C^{\alpha}(\R^3)$ at least on $[0,T_n]$, with $u_n$ satisfying the estimate (see \cite{Chemin2} and chapter 4 of \cite{Chemin1})  
\begin{equation}\label{EHolderuniformbound}
\begin{split}
&\| u_n \|_{C([0,T_n]; C^{\alpha})} \leq \frac{\| u_n^0 \|_{C^{\alpha}}}{1-CT_n\| u_n^0 \|_{C^{\alpha}}}\leq \frac{\tilde{C}\| u_n^0 \|_{H^s_{ul}}}{1-\tilde{C}T_n\| u_n^0 \|_{H^s_{ul}}}.\\ %
\end{split}
\end{equation}    
Choose $T>0$ such that, for every $n$, $T\leq T_n$ and satisfies
\begin{equation*}
\frac{1}{2\tilde{C}} \leq T_n \| u_n^0 \|_{H^s_{ul}} \leq T \| u^0 \|_{H^s_{ul}} < \frac{1}{\tilde{C}},
\end{equation*}
where $\tilde{C}$ is as in (\ref{EHolderuniformbound}).  We have that for every $n$, $u_n$ is a solution to ($E$) in $C^{\alpha}(\R^3)$ on $[0,T]$.  Moreover, by (\ref{EHolderuniformbound}) and Lemma \ref{unifbound}, $\| u_n \|_{C([0,T]; C^{\alpha})}$ is uniformly bounded in $n$.  But this implies that for every $n$, $\| \nabla u_n \|_{L^{\infty}(0,T; L^{\infty}(\R^3)) } < \infty$.  From this and classical theory we can conclude that $u_n$ belongs to $C([0,T];H^k(\R^3))$ for every $k$.  Thus, for every $n$, $u_n$ satisfies Theorem \ref{mainLpversion} on $[0,T]$.  Theorem \ref{mainLpversion}, Lemma \ref{unifbound}, and another application of the uniform bound on $\| u_n \|_{C([0,T]; C^{\alpha})}$ imply that there exists a constant $C>0$, depending only on the initial data and $T$, such that for all $n$,
\begin{equation}\label{Eulervortunifbound}
\| \omega_n \|_{C([0,T]; H^{s-1}_{ul})} \leq C.
\end{equation}    
Moreover, by (\ref{uomegasobolevbound}), (\ref{Eulervortunifbound}), and (\ref{EHolderuniformbound}), there exists a constant $C>0$, depending only on the initial data and $T$, such that for all $n$,
\begin{equation}\label{Eulervelunifbound}
\| u_n \|_{C([0,T]; H^{s}_{ul})} \leq C.
\end{equation} 

\noindent\textbf{$\bm{(u_n)}$ converges to $\bm{u}$.}
Note that for each $n$, $u_n$ belongs to the space $C^1([0,T]; H^{s}(\R^3))$.  Moreover, we have 
\begin{equation}\label{Eulern}
\partial_t u_n + u_n \cdot \nabla u_n = -\nabla p_n,
\end{equation}
where $p_n$ satisfies $ p_n = \Delta^{-1}\nabla (u_n\cdot\nabla u_n)$. 
For fixed $R>0$, multiply (\ref{Eulern}) by $\phi_R$.  Then
\begin{equation*}
\| \phi_R \partial_t u_n \|_{H^{s-1}} \leq \| \phi_R( u_n\cdot\nabla u_n) \|_{H^{s-1}} + \| \phi_R\nabla p_n \|_{H^{s-1}}.
\end{equation*}  
Note that 
\begin{equation*}
\begin{split}
&\| \phi_R( u_n\cdot\nabla u_n) \|_{H^{s-1}} \leq C(R)\|  u_n\cdot\nabla u_n \|_{H^{s-1}_{ul}} \leq C(R)\| u_n \|_{L^{\infty}} \| \nabla u_n\|_{H^{s-1}_{ul}} \\
&\qquad +  C(R)\|\nabla u_n \|_{L^{\infty}} \|  u_n\|_{H^{s-1}_{ul}} \leq C(R) \| u_n \|^2_{H^s_{ul}},
\end{split}
\end{equation*}
which can be bounded uniformly in $n$ by (\ref{Eulervelunifbound}).

To estimate the pressure term, observe that
\begin{align}\label{smoothpidentity}
\begin{split}
&\nabla p_n = - (a\grad G)\ast \dv \dv (u_n \otimes u_n)  + [\nabla\nabla((1-a)\grad G)]\cdot\ast(u_n \otimes u_n).
\end{split}
\end{align}
Applying $D^{\gamma}$ to this identity with $1\leq |\gamma| \leq s-1$ and applying Lemma \ref{L:ConvHsBound},
\begin{equation*}
\| \nabla p_n \|_{H^{s-1}_{ul}} \leq C(\| u_n \otimes u_n \|_{H^s_{ul}} + \| u_n \otimes u_n \|_{L^{\infty}}) \leq C\| u_n \otimes u_n \|_{H^s_{ul}},
\end{equation*}
where we used the Sobolev embedding theorem.  Thus, 
\begin{equation}\label{pressureunifbound}
\| \phi_R\nabla p_n \|_{H^{s-1}} \leq C(R)\| \nabla p_n \|_{H^{s-1}_{ul}} \leq C\| u_n \|^2_{H^s_{ul}},
\end{equation}
which can be uniformly bounded in $n$.

Combining the above inequalities, we conclude that
\begin{equation}\label{continuityint}
\| \phi_R \partial_t u_n \|_{H^{s-1}} \leq C,
\end{equation}
with $C$ depending on the initial data and $R$, but not on $n$.  

By Rellich's Theorem and the uniform bounds on $\| u_n (t) \|_{H^{s}_{ul}}$ for each $t$, we can conclude that for each $t$ and each $R$, there exists a subsequence of $(\phi_Ru_n(t))$ which converges in $H^{s-1}(\R^3)$.  Using a standard diagonalization argument, for each fixed $t$, one can find a subsequence of $(\phi_Ru_n(t))$, relabelled $(\phi_Ru_n(t))$, which converges in $H^{s-1}(\R^3)$ for every $R$.

To find a single subsequence that works for all $t$, we use (\ref{continuityint}). Given $\epsilon>0$, there exists $\delta>0$ such that for all $n$,  
\begin{equation}\label{velunifctsSobolev}
\| \phi_R  u_n(s) - \phi_R  u_n(t) \|_{H^{s-1}} < \epsilon/3
\end{equation}          
whenever $|s-t|< \delta$.  Given this $\delta$, construct a partition of $[0,T]$, $0=t_0 < t_1 < t_2 < ...... < t_M = T$ such that $t_i - t_{i-1} <\delta$.  Using the process above, one can find a subsequence of $(\phi_R u_n)$, which we relabel $(\phi_R u_n)$, such that $(\phi_R u_n(t_i))$ converges, and hence is Cauchy in, $H^{s-1}$ for each $t_i$, $i=1,2,...,M$ and for every $R>0$.  

Let $N$ be such that for all $n,m\geq N$ and for all $t_i$ in the partition,
\begin{equation*}
\| \phi_R  u_n(t_i) - \phi_R  u_m(t_i) \|_{H^{s-1}} < \epsilon/3.
\end{equation*}        
Then for all pairs $m$, $n \geq N$ and for each $t\in[0,T]$, there exists $t_i$ such that   
\begin{equation*}
\begin{split}
&\|  \phi_R  u_n(t) -  \phi_R  u_m(t) \|_{H^{s-1}} \leq \| \phi_R u_n(t) - \phi_R u_n(t_i) \|_{H^{s-1}}\\
&\qquad   + \| \phi_R u_n(t_i) - \phi_R u_m(t_i) \|_{H^{s-1}}  + \| \phi_R u_m(t_i) - \phi_R u_m(t) \|_{H^{s-1}} < \epsilon. 
\end{split}
\end{equation*}
We conclude that $(\phi_R u_n)$ is Cauchy in $C([0,T]; H^{s-1}(\R^3))$, and thus there exists $u$ such that $(\phi_R u_n)$ converges to $\phi_R u$ in $C([0,T]; H^{s-1}(\R^3))$ for all $R>0$.  \\

\noindent\textbf{$\bm{(p_n)}$ converges to $\bm{p}$.}
We now show that, up to subsequences, for all $R>0$, $(\phi_R\nabla p_n)$ is Cauchy (and thus converges) in $C([0,T];H^{s-2}(\R^3))$.  The process is very similar to that above.  As above, using the uniform bound in (\ref{pressureunifbound}) and Rellich's Theorem, we can conclude that for each fixed $t$, there exists a subsequence of $(\phi_R\nabla p_n(t))$, relabel it $(\phi_R\nabla p_n(t))$, which converges in $H^{s-2}(\R^3)$ for every $R$.  To find a single subsequence that works for all $t$, we must find a time modulus of continuity for $(  \phi_R  \nabla p_n(t) )$ which is uniform in $n$.  To do this, first note that by Proposition \ref{P:PressureIdentity},
\begin{equation*}
\begin{split}
&\| \partial_t u_n \|_{L^{\infty}} \leq \| u_n \|_{L^{\infty}} \| \nabla u_n \|_{L^{\infty}} + \| \nabla p_n \|_{L^{\infty}} \\
&\qquad \leq C\| u_n \|^2_{\tilde{C}^1} \leq C\| u_n \|^2_{H^s_{ul}} \leq C 
\end{split}
\end{equation*}
for all $n$ and for all $t\in [0,T]$.  Thus there exists $C>0$ such that for all $s,t\in [0,T]$ and for all $n$,
\begin{equation}\label{velunifcontinuity}
\| u_n(t) - u_n (s) \|_{L^{\infty}} \leq C|t-s|. 
\end{equation}

Applying (\ref{smoothpidentity}) and Lemma \ref{L:ConvHsBound},
for $s,t\in[0,T]$,  
\begin{equation*}
\begin{split}
&\|  \phi_R  \nabla p_n(t) -  \phi_R  \nabla p_n(s)\|_{H^{s-2}} \leq C\| \phi_{4R} (u_n \otimes u_n(t) - u_n \otimes u_n(s) )\|_{H^{s-1}}\\
&\qquad + C\| u_n \otimes u_n(t) - u_n \otimes u_n(s) \|_{L^{\infty}}.
\end{split}
\end{equation*}     
It follows from uniform bounds on $\| u_n \|_{L^{\infty}}$ and $\| u_n \|_{H^{s-1}_{ul}}$ in $n$, along with (\ref{velunifcontinuity}) and (\ref{velunifctsSobolev}), that given $\epsilon>0$, there exists $\delta>0$ such that for all $n$, whenever $|s-t|<\delta$, 
\begin{equation*}
\|  \phi_R  \nabla p_n(t) -  \phi_R  \nabla p_n(s)\|_{H^{s-2}} < \epsilon.
\end{equation*}

With this uniform continuity in hand, we follow a process identical to that used to show for all $R>0$, $(\phi_R u_n)$ is Cauchy in $C([0,T];H^{s-1}(\R^3))$.  We conclude that for all $R>0$, $(\phi_R\nabla p_n)$ is Cauchy in $C([0,T];H^{s-2}(\R^3))$, and thus there exists $p$ such that $(\phi_R\nabla p_n)$ converges to $\phi_R \grad p$ in $C([0,T];H^{s-2}(\R^3))$. \\

\noindent\textbf{$\bm{(u, p)}$ solve $\bm{(E)}$}.
For fixed $R>0$, multiply (\ref{Eulern}) by $\phi_R$.  Then for any $m, n$, %
\begin{equation*}
\begin{split}
&\phi_R(\partial_t u_n - \partial_t u_m) = \phi_R(u_n - u_m)\cdot\nabla u_m + \phi_Ru_n\cdot\nabla(u_m - u_n) - \phi_R\nabla (p_n - p_m)\\
&= \phi_R(u_n - u_m)\phi_{2R}\cdot\nabla u_m + \phi_Ru_n\cdot\phi_{2R}\nabla(u_m - u_n) - \phi_R\nabla (p_n - p_m),
\end{split}
\end{equation*}  
so that, for each $t$,
\begin{equation*}
\begin{split}
&\| \phi_R(\partial_t u_n - \partial_t u_m)\|_{H^{s-2}} \leq \| \phi_R(u_n - u_m)\phi_{2R}\nabla u_m \|_{H^{s-2}} \\
&\qquad + \| \phi_Ru_n\cdot\phi_{2R}\nabla(u_m - u_n) \|_{H^{s-2}} + \|\phi_R\nabla (p_n - p_m)\|_{H^{s-2}}\\
&\qquad \leq \| \phi_R(u_n - u_m)\|_{L^{\infty}} \|\phi_{2R}\nabla u_m \|_{H^{s-2}} + \| \phi_R(u_n - u_m)\|_{H^{s-2}} \|\phi_{2R}\nabla u_m \|_{L^{\infty}}\\
&\qquad + \| \phi_Ru_n\|_{L^{\infty}} \|\phi_{2R}\nabla(u_m - u_n) \|_{H^{s-2}} + \| \phi_Ru_n\|_{H^{s-2}} \|\phi_{2R}\nabla(u_m - u_n) \|_{L^{\infty}} \\
&\qquad + \|\phi_R\nabla (p_n - p_m)\|_{H^{s-2}}.
\end{split}
\end{equation*}  
Note that $\|\phi_{2R}\nabla u_n \|_{H^{s-2}}$, $\|\phi_{2R}\nabla u_n \|_{L^{\infty}}$, $\| \phi_Ru_n\|_{L^{\infty}}$, and $\| \phi_Ru_n\|_{H^{s-2}}$ are uniformly bounded in $n$.  We conclude that as $N \to \iny$,
\begin{equation*} 
\begin{split}
& \sup_{m, n \ge N} \| \phi_R(u_n - u_m)\|_{L^{\infty}} \|\phi_{2R}\nabla u_m \|_{H^{s-2}} \rightarrow 0,\\
& \sup_{m, n \ge N} \| \phi_R(u_n - u_m)\|_{H^{s-2}} \|\phi_{2R}\nabla u_m \|_{L^{\infty}} \rightarrow 0,\\
& \sup_{m, n \ge N} \| \phi_Ru_n\|_{L^{\infty}} \|\phi_{2R}\nabla(u_m - u_n) \|_{H^{s-2}} \rightarrow 0, \\
& \sup_{m, n \ge N} \| \phi_Ru_n\|_{H^{s-2}} \|\phi_{2R}\nabla(u_m - u_n) \|_{L^{\infty}} \rightarrow 0, \\
& \sup_{m, n \ge N} \| \phi_R \nabla (p_n-p_m)\|_{H^{s-2}}\rightarrow 0.
\end{split}
\end{equation*}

From the estimates above, it follows that $(\phi_R\partial_t u_n)$ is Cauchy in $C([0,T];H^{s-2}(\R^3))$.  Since $\phi_R u_n\rightarrow \phi_R u$ in $C([0,T]\times \R^3)$, $\phi_R u_n\rightarrow \phi_R u$ in ${\mathcal{D'}}([0,T]\times\R^3)$, which means $\phi_R \partial_t u_n\rightarrow \phi_R \partial_t u$ in ${\mathcal{D'}}([0,T]\times\R^3)$.  Thus, by uniqueness of weak limits, $\phi_R \partial_t u_n\rightarrow \phi_R \partial_t u$ in $C([0,T];H^{s-2}(\R^3))$ for every $R$.  This, combined with convergence of $(\phi_R u_n)$ to $(\phi_R u)$ in $C([0,T];H^{s-1}(\R^3))$ for every $R$, allows us to conclude that for every $R>0$, 
\begin{equation*}
\begin{split}
&\phi_R\partial_t u_n \rightarrow \phi_R\partial_t u, \\
&\phi_Ru_n\cdot\nabla u_n \rightarrow \phi_Ru\cdot\nabla u
\end{split}
\end{equation*}
in $C([0,T];H^{s-2}(\R^3))$.

It remains to take the limit of $(\phi_R\nabla p_n)$ in $C([0,T];H^{s-2}(\R^3))$.  To do this, first note that by Proposition \ref{P:PressureIdentity}, for every $n$,   
\begin{equation*}
\begin{split}
\nabla p_n(t,x)  &= - \int_{\R^3} a(x - y)  \nabla G (x - y)
				        \dv \dv (u_n \otimes u_n)(t,y) \, dy \\
			    &\qquad
	 			    + \int_{\R^3} (u_n \otimes u_n)(t,y) \cdot
					    \grad \grad \brac{(1 - a(x - y)) \nabla G(x - y)}
					    \, dy.
	    \end{split}
\end{equation*} 
Since $( u_n\otimes u_n )$ is uniformly bounded in $C([0,T]; L^{\infty}(\R^3))$, for each $t\in[0,T]$, there exists a subsequence $(u_{n_k} (t)\otimes u_{n_k}(t))$ converging weak-* in $L^{\infty}(\R^3)$.  Since $(\phi_R u_n)$ converges to $\phi_R u$ in $C([0,T];H^{s-1}(\R^3))$ for each $R$, $(\phi_R u_n\otimes \phi_R u_n)$ converges to $\phi_R u\otimes \phi_R u$ in $C([0,T];H^{s-1}(\R^3))$ for each $R$.  It follows from uniqueness of weak limits that for this fixed $t\in[0,T]$, $( u_{n_k}(t)\otimes u_{n_k}(t) )$ converges weak-* in $L^{\infty}$ to $u(t)\otimes u(t)$.  Since, for each $x\in\R^3$, $\grad \grad \brac{(1 - a(x - y)) \nabla G(x - y)}$ is in $L_y^1(\R^3)$,  
\begin{equation*}
\begin{split}
&\int_{\R^3} (u_{n_k} \otimes u_{n_k})(t,y) \cdot
					    \grad \grad \brac{(1 - a(x - y)) \nabla G(x - y)}
					    \, dy\\
&\qquad \rightarrow \int_{\R^3} (u \otimes u)(t,y) \cdot
					    \grad \grad \brac{(1 - a(x - y)) \nabla G(x - y)}
					    \, dy
\end{split}
\end{equation*}
for each $x\in\R^3$.

Similarly, since $(\dv \dv (u_n \otimes u_n) ) = ( \nabla u_n\cdot (\nabla u_n )^T)$ is uniformly bounded in the space $C([0,T]; L^{\infty}(\R^3))$, for each $t\in[0,T]$, there exists a subsequence $(\dv \dv (u_{n_k}(t) \otimes u_{n_k}(t)) )$ converging weak-* in $L^{\infty}(\R^3)$.  But again, since $(\phi_R u_n\otimes \phi_R u_n)$ converges to $\phi_R u\otimes \phi_R u$ in $C([0,T];H^{s-1}(\R^3))$ for each $R$, for this fixed $t\in[0,T]$, $(\dv \dv( u_{n_k}\otimes u_{n_k}))$ converges to $\dv \dv( u\otimes  u)$ in $\mathcal{D}'(\R^3)$.  By uniqueness of weak limits, the weak-* limit of $(\dv \dv (u_{n_k}(t) \otimes u_{n_k}(t)) )$ must be $(\dv \dv (u(t) \otimes u(t)) )$.  Since, for each $x\in\R^3$, $a(x - y)\nabla G (x - y)$ is in $L_y^1(\R^3)$,  
\begin{equation*}
\begin{split}
&\int_{\R^3} a(x - y)  \nabla G (x - y)
				        \dv \dv (u_{n_k} \otimes u_{n_k})(t,y) \, dy\\
				        &\qquad  \rightarrow \int_{\R^3} a(x - y)  \nabla G (x - y)
				        \dv \dv (u \otimes u)(t,y) \, dy
\end{split}
\end{equation*}
for each $x\in\R^3$. 

We conclude that for each $t\in[0,T]$, there exists a subsequence $(\nabla p_{n_k})$ such that $\nabla p_{n_k}(t,x) \rightarrow \nabla p(t,x)$ for every $x\in\R^3$, where
\begin{equation}\label{pressureforsolution}
\begin{split}
\nabla p(t,x)  &= - \int_{\R^3} a(x - y)  \nabla G (x - y)
				        \dv \dv (u \otimes u)(t, y) \, dy \\
			    &\qquad
	 			    + \int_{\R^3} (u \otimes u)(t, y) \cdot
					    \grad \grad \brac{(1 - a(x - y)) \nabla G(x - y)}
					    \, dy.
\end{split}
\end{equation}

Finally, since $(\phi_R \nabla p_n)$ converges in $C([0,T]; H^{s-2}(\R^3))$, by the above it must converge to $\phi_R \nabla p$ in $C([0,T]; H^{s-2}(\R^3))$.  Thus, $(u,p)$ solves ($E$), where $p$ satisfies (\ref{pressureforsolution}). \\

\noindent\textbf{$\bm{u}$ belongs to $\bm{L^{\infty}(0,T; H^s_{ul}(\R^3))}$}.
By (\ref{Eulervelunifbound}), for every $x\in\R^2$, $n\in\N$, and $t\in[0,T]$, $$\| \phi_x u_n(t) \|_{H^s} \leq C.$$  Therefore, up to a subsequence which depends on $t$ and $x$, $\phi_x u_n(t)$ converges weak-* in $H^s(\R^2)$.  Note, however, that for every $R>0$ and $t\in[0,T]$, $\phi_Ru_n(t)\rightarrow \phi_Ru(t)$ in $H^{s-1}(\R^2)$.  Given $x$, since we can always choose $R$ large enough to ensure that $\phi_x = \phi_x\phi_R$, we have $\phi_xu_n(t)\rightarrow \phi_xu(t)$ in $L^2(\R^2)$.  By uniqueness of limits, $\phi_x u_n(t)$ converges weak-* in $H^s(\R^2)$ to $\phi_x u(t)$, and 
$$ \| \phi_x u(t) \|_{H^s} \leq C.$$
This holds for all $t\in [0,T]$ and for all $x\in\R^2$, so $u$ belongs to $L^{\infty}(0,T; H^s_{ul}(\R^2))$. \\

\noindent \textbf{Uniqueness} Applying a cutoff function $\phi_R$ to two solutions and making the same estimates that showed $(u, p)$ solve $(E)$ yields uniqueness. Moreover, uniqueness also follows from \cite{TTY2010} or from \cite{CozziKelliher}. 
\end{proof}

\appendix
\section{A constitutive relation for ($SQG$)}\label{SQGSerfatiDerivation}

\begin{lemma}\label{L:SerfatiSQG}
	Assume that $(u, \theta)$ are smooth solutions on $[0, T] \times \R^2$, with $\theta$
	compactly supported in space, to 
	\begin{align*}
		\begin{cases}
			\prt_t \theta + u \cdot \grad \theta = 0, \\
			u(t) = \grad^\perp (\Delta)^{-\frac{1}{2}} \theta(t), \\
			(u, \theta)|_{t = 0} = (u^0, \theta^0).
		\end{cases}
	\end{align*}
	Then for all $t \in [0, T]$ and any $\lambda > 0$, we have the Serfati-type identity,
	\begin{align*}
		u(t)
			&= u^0
				+ (a_\lambda \Phi) * \grad^\perp (\theta(t) - \theta^0)
					- \int_0^t (\grad \grad^\perp ((1 - a_\lambda) \Phi) \stardot (\theta u(s))
						\, ds.
	\end{align*}
	In indices, this is
	\begin{align*}
		u^i(t)
			&= (u^0)^i + (a\Phi) * (\nabla^{\perp}(\theta(t) - \theta^0))^i
				- \int_0^t \prt_j (\grad^\perp((1-a)\Phi))^i * (\theta u^j(s)) \, ds.
	\end{align*}
\end{lemma}
\begin{proof}
	Because $\theta$ is compactly supported in space, we can write the constitutive law in the form
	$
		u(t) = \grad^\perp (\Phi * \theta(t)).
	$
	Taking the time derivative, we can introduce the cutoff function to obtain
	\begin{align*}
		\prt_t u(t)
			&= \grad^\perp (\Phi * \prt_t \theta(t))
			= \grad^\perp ((a_\lambda \Phi) * \prt_t \theta(t))
				+ \grad^\perp (((1 - a_\lambda) \Phi) * \prt_t \theta(t)) \\
			&= \prt_t  ((a_\lambda \Phi) * \grad^\perp \theta(t))
				- \grad^\perp (((1 - a_\lambda) \Phi) * (u \cdot \grad \theta)(t)).
	\end{align*}
	But $u \cdot \grad \theta = \dv (\theta u)$, so
	\begin{align*}
		&\brac{\grad^\perp (((1 - a_\lambda) \Phi) * (u \cdot \grad \theta)(t))}^i
			= \brac{\grad^\perp (((1 - a_\lambda) \Phi)}^i * (\dv (\theta u)(t)) \\
			&\qquad
			= \grad \brac{\grad^\perp (((1 - a_\lambda) \Phi)}^i \stardot (\theta u)(t).
	\end{align*}
	Integrating in time completes the proof.
\end{proof}

\begin{lemma}\label{BMOSerfati}
Assume the sequences ($u_n$) and ($\theta_n$) are generated as in (\ref{iteration}) and (\ref{seqapprox}).  For every $j\in\Z$, $n\in\N$, and $t\in [0,T]$, 
\begin{equation*}
\dot{\Delta}_j u^n (t) = \dot{\Delta}_j \nabla^{\perp}(-\Delta)^{-1/2} \theta^n (t),
\end{equation*}
equality holding almost everywhere on $\R^2$.
\end{lemma}
\begin{proof}
Applying $\prt_t$ to $(\ref{seqapprox})_3$ gives, for every $j\in\Z$,
\begin{equation}\label{timeder}
\begin{split}
&\varphi_j\ast \partial_t u^{n}(t) =  \varphi_j\ast  ( (a\Phi)\ast\partial_t\nabla^{\perp}\theta^{n}(t) ) -  \varphi_j\ast  \left(  \nabla L \stardot (u^{n-1}\theta^{n})(t)\right),
\end{split}
\end{equation}
where $L = \nabla^\perp((1-a)\Phi)$, which we note has the singularity at the origin removed and which decays like $C \abs{x}^{-2}$ as $x \to \iny$. We apply the Fourier transform to both sides of (\ref{timeder}).  This gives
\begin{equation*}
\begin{split}
	&\hat{\varphi}_j \mathcal{F}(\partial_t u^{n})
		=  \hat{\varphi}_j \mathcal{F}(a\Phi)\mathcal{F}(\partial_t\nabla^{\perp}\theta^{n})
			- \hat{\varphi}_j  \mathcal{F} ( \nabla L ) \mathcal{F}(u^{n-1}\theta^{n}) \\
		&\qquad
		= \hat{\varphi}_j(\hat{a} \ast \hat{\Phi})(i\xi^\perp)\mathcal{F}(\partial_t\theta^{n})
			- \hat{\varphi}_j  (i\xi^\perp)
				(\mathcal{F}(1 - a)\ast \hat{\Phi})
				\brac{i\xi \cdot \mathcal{F}(u^{n-1}\theta^{n})} \\
		&\qquad
		= i \hat{\varphi}_j \xi^\perp
			\brac{(\hat{a} \ast \hat{\Phi}) \mathcal{F}(\partial_t\theta^{n})
			- (\mathcal{F}(1 - a)\ast \hat{\Phi})
				\brac{i\xi \cdot \mathcal{F}(u^{n-1}\theta^{n})}
			}.
\end{split}
\end{equation*}

But,
\begin{align*}
	i\xi \cdot \mathcal{F}(u^{n-1}\theta^{n})
		&= \mathcal{F}(\dv (\theta^n u^{n - 1}))
		= \mathcal{F}(u^{n - 1} \cdot \grad \theta^n)
		= -\mathcal{F}(\prt_t \theta^n),
\end{align*}
so
\begin{align*}
	\hat{\varphi}_j \mathcal{F}(\partial_t u^{n})
		&= i \mathcal{F}(\partial_t\theta^{n}) \hat{\varphi}_j \xi^\perp
			\brac{(\hat{a} \ast \hat{\Phi}) 
			+ (\mathcal{F}(1 - a)\ast \hat{\Phi})
			}.
\end{align*}

Note that $\hat{a}\in\mathcal{S}$, and $\hat{\Phi}$ decays like $\abs{\xi}^{-1}$, so that $\hat{a} \ast \hat{\Phi} = \hat{a}\ast(a \hat{\Phi} ) + \hat{a}\ast((1-a)\hat{\Phi})$ is in $L^1 + L^p$ for all $p>2$, by Young's inequality.  Moreover, observe that
\begin{equation}\label{highmodes}
	\hat{\varphi}_j (\mathcal{F}(1 - a)\ast \hat{\Phi})
		= \hat{\varphi}_j ((\delta - \hat{a}) \ast \hat{\Phi})
		= \hat{\varphi}_j  \hat{\Phi} - \hat{\varphi}_j  (\hat{a}\ast \hat{\Phi}).
\end{equation}
Since $\hat{\varphi}_j  \hat{\Phi} \in \mathcal{S}$, we have that $\hat{\varphi}_j (\mathcal{F}(1 - a)\ast \hat{\Phi})$ belongs to $L^1 + L^p$ as well.  In particular, all three terms in (\ref{highmodes}) are defined almost everywhere as, then, are the products.  This allows us to write the following equality, which holds in the distributional sense:
\begin{equation}\label{Fourierconstitutivelaw}
\begin{split}
	&\hat{\varphi}_j \mathcal{F}(\partial_t u^{n})
		= \hat{\varphi}_j(i\xi^\perp )\hat{\Phi}\mathcal{F}(\partial_t\theta^{n}).
\end{split}
\end{equation}
Defining $G \in \mathcal{S}(\R^2)$ by $G = \mathcal{F}^{-1} [ \hat{\varphi_j}(i\xi^\perp) \hat{\Phi} ]$ and applying the inverse Fourier transform to (\ref{Fourierconstitutivelaw}) gives
\begin{equation}\label{convolutionconstitutivelaw}
{\varphi}_j \ast \partial_t u^{n} = G \ast \partial_t\theta^{n} \qquad \text{ in } \mathcal{S}'(\R^2).
\end{equation}
Since both sides of the equality in (\ref{convolutionconstitutivelaw}) are convolutions of Schwarz functions with bounded functions, both sides belong to $L^1_{loc}(\R^2)$.  Therefore, equality in (\ref{convolutionconstitutivelaw}) holds pointwise almost everywhere on $\R^2$.  Moreover, by (\ref{LPRieszTransform}), we can write
\begin{equation*}
\dot{\Delta}_j \partial_t u^{n}(t) = \dot{\Delta}_j\nabla^{\perp}(-\Delta)^{-1/2} \partial_t\theta^{n},
\end{equation*}
which also holds almost everywhere.
Integrating in time and using the identity $\dot{\Delta}_ju^0 = \dot{\Delta}_j\nabla^{\perp}(-\Delta)^{-1/2} \theta^0$ for all $j\in\Z$, we have that for all $t\geq 0$,
\begin{equation*}
\dot{\Delta}_j u^{n}(t) = \dot{\Delta}_j \nabla^{\perp}(-\Delta)^{-1/2} \theta^{n},
\end{equation*} 
proving the lemma.
\end{proof}

%
%
\section{Serfati identity for 3D Euler}

We establish the 3D version of the Serfati identity of \cite{Serfati}. The key point of this identity is not its precise form, but rather the order of the derivatives that appear on its near and far field terms.

\begin{lemma}\label{L:SertatiID3DEuler}
	Let
	\begin{align*}
		K(x) = \frac{x}{4 \pi \abs{x}^3},
	\end{align*}
	one form of the 3D Biot-Savart kernel.
	Any smooth solution to the 3D Euler equations with velocity $u$ and with vorticity
	$\omega$ compactly supported in space satisfies, for any $\lambda > 0$,
	the 3D Serfati identity,
	\begin{align*}
		u^k(t)
			= (u^0)^k
				&+ \int_{\R^3} (a_\lambda K)(x - y) \times \omega(t, y) \, dy \\
				&+ \int_{\R^3} \grad \grad ((1 - a_\lambda) K^k)(x - y) \stardot
					(u \otimes u) (t, y) \, dy \\
				&+ \int_{\R^3} \grad \dv ((1 - a_\lambda) K)(x - y) \stardot
					(u^k \, u) (t, y) \, dy.
	\end{align*}
\end{lemma}
\begin{proof}
	Because $\omega$ is compactly supported in space, we can write the constitutive law
	in the form (see, for instance, Proposition 2.16 of \cite{MB})
	\begin{align*}
		u(t, x)
			&= \int_{\R^3} K(x - y) \times \omega(t, x) \, dy.
	\end{align*}
	
	Proceeding as in the proof of (\ref{L:SerfatiSQG}), we have
	\begin{align*}
		\prt_t u(t)
			&= \diff{}{t} \int_{\R^3} (a_\lambda K)(x - y) \times \omega(t, y) \, dy
				+ \int_{\R^3} L(x - y) \times \prt_t \omega(t, y) \, dy,
	\end{align*}
	where $L = (1 - a_\lambda) K$.
	But $\prt_t \omega = \curl (\prt_t u) = - \curl (u \cdot \grad u) = - \curl \dv (u \otimes u)$.
	Hence, the $i^{th}$ component of the second integral above, using Lemma \ref{L:IBPEuler}, becomes
	\begin{align*}
		- &\brac{\int_{\R^3} L(x - y) \times \curl \dv (u \otimes u)(t, y) \, dy}^k \\
			&\quad
			= \int_{\R^3} \prt_{y_i} L^k(x - y) \cdot [\dv (u \otimes u)(t, y)]^i
				+ \prt_{y_i} L^i(x - y) [\dv (u \otimes u)(t, y)]^k \, dy \\
			&\quad
			= \int_{\R^3} \prt_{y_i} L^k(x - y) \prt_j (u^j \otimes u^i)(t, y)
				+ \prt_{y_i} L^i(x - y) \prt_j (u^j \otimes u^k)(t, y) \, dy \\
			&\quad
			= -\int_{\R^3} \prt_j \prt_i L^k(x - y) (u^j \otimes u^i)(t, y)
				+ \prt_j \prt_i L^i(x - y) (u^j \otimes u^k)(t, y) \, dy.
	\end{align*}
	Integrating in time yields the result.
\end{proof}

\begin{lemma}\label{L:IBPEuler}
	For $u, v$ smooth with $u v$ compactly supported,
	\begin{align*}
		\int_{\R^3} u \times \curl v
			&= \int_{\R^3} (- \grad u \cdot v + \dv u \, v)
			= \int_{\R^3} (- \prt_i u^k v^i + \prt_i u^i \, v^k) e_k.
	\end{align*}
\end{lemma}
\begin{proof}
	We have,
	\begin{align*}
		u \times \curl v
			&=
			\begin{vmatrix}
				\bf{i} & \bf{j} & \bf{k} \\
				u^1 & u^2 & u^3 \\
				\prt_2 v^3 - \prt_3 v^2 & \prt_3 v^1 - \prt_1 v^3 & \prt_1 v^2 - \prt_2 v^1
			\end{vmatrix}.
	\end{align*}
	Working only on the first component and integrating by parts, we have
	\begin{align*}
		\int_{\R^3} (u \times \curl v)^1
			&= \int_{\R^3} u^2 (\prt_1 v^2 - \prt_2 v^1) - u^3(\prt_3 v^1 - \prt_1 v^3) \\
			&= \int_{\R^3} (- \prt_1 u^2 v^2 + \prt_2 u^2 v^1 + \prt_3 u^3 v^1 - \prt_1 u^3 v^3).
	\end{align*}
	But $\prt_2 u^2 + \prt_3 u^3 = \dv u - \prt_1 u^1$, so
	\begin{align*}
		\int_{\R^3} (u \times \curl v)^1
			&= \int_{\R^3} (- \prt_1 u^2 v^2 + (\dv u - \prt_1 u^1) v^1 - \prt_1 u^3 v^3) \\
			&= \int_{\R^3} (- \prt_1 u \cdot v + \dv u \, v^1).
	\end{align*}
	Similar expressions for the other two terms give the result.
\end{proof}

\section{Pressure identity}

We derive in this appendix the pressure identity for solutions to the Euler equations, adapted from the 2D version due to Serfati \cite{Serfati}, as derived in \cite{Kelliher}.

We work throughout with a sufficiently smooth decaying solution, $(u, p)$, to the 3D Euler equations in all of $\R^3$. It is classical in that setting that
\begin{align}\label{e:pqnR2}
	p(t, x)
		&= - G*\dv \dv (u(t) \otimes u(t))(x),
\end{align}
where $G$ is the fundamental solution to the Laplacian on $\R^3$, defined in (\ref{e:G}).

\begin{prop}\label{P:PressureIdentity}
	Let $a$ be as in Section \ref{preliminary}.
    The identity,
    \begin{align}\label{e:gradpExpDup}
    \begin{split}
        \grad p(x)
			&= - \int_{\R^3} a(x - y) \grad G (x - y)
				        \dv \dv (u \otimes u)(y) \, dy \\
			    &\qquad
	 			    + \int_{\R^3} (u \otimes u)(y) \cdot
					    \grad \grad \brac{(1 - a(x - y)) \grad G(x - y)}
					    \, dy,
	    \end{split}
    \end{align}
    holds independently of the choice of cutoff function,
    and $\grad p \in L^\iny([0, T] \times \R^3)$ with
    \begin{align*}
		\norm{\grad p(t)}_{L^\iny}
			\le C \norm{u(t)}_{\tilde{C}^1}^2.
	\end{align*}
\end{prop}
\begin{proof}
Applying $\prt_i$ to (\ref{e:pqnR2}) gives
\begin{align*}
	\prt_i p(x)
		=  -\int_{\R^3} \prt_i G(x - y) \dv(u \cdot \grad u)(y) \, dy.
\end{align*}
Here, we suppress the time variable to streamline notation.
Applying a cutoff and integrating by parts,
\begin{align*}
	\prt_i p(x)
		&= - \int_{\R^3} a(x - y) \prt_i G(x - y) \dv(u \cdot \grad u)(y) \, dy \\
		&\qquad
			- \int_{\R^3} (1 - a(x - y)) \prt_i G(x - y)
			        \dv(u \cdot \grad u)(y) \, dy \\
		&= - \int_{\R^3} a(x - y) \prt_i G(x - y) \dv(u \cdot \grad u)(y) \, dy \\
		&\qquad
	 		+ \int_{\R^3} (u \cdot \grad u)(y)
			\cdot \grad \brac{(1 - a(x - y)) \prt_i G(x - y)} \, dy.
\end{align*}
Integrating as in Lemma \ref{L:SimpleLemma} gives
\begin{align*}
	\begin{split}
	\prt_i p(x)
		&= - \int_{\R^3} a(x - y) \prt_i G(x - y)
		    \dv(u \cdot \grad u)(y) \, dy \\
		&\qquad
	 		+ \int_{\R^3} (u(y) \cdot \grad_y)
			\grad_y \brac{(1 - a(x - y)) \prt_i G(x - y)}
			\cdot u(y) \, dy,
	\end{split}
\end{align*}
which we can write more succinctly as (\ref{e:gradpExpDup}).

We conclude, since $\dv (u \cdot \grad u) = \grad u \cdot (\grad u)^T$, that
\begin{align*} 
	\begin{split}
	\norm{\prt_i p}_{L^\iny}
		&\le \smallnorm{a \prt_i G}_{L^1}
			\norm{\grad u}_{L^{\infty}}^2
			+ \smallnorm{\grad \grad \brac{(1 - a) \prt_i G}}_{L^1}
				\norm{u}_{L^\iny}^2.
	\end{split}
\end{align*}
Here, we are using that (in any dimension), $\grad G$ is locally in $L^1$ and, away from its singularity, $\grad^3 G$ lies in $L^1$.
This gives the bound on $\grad p(t)$ in $L^\iny$.

That the expression in (\ref{e:gradpExpDup}) is independent of the choice of cutoff function $a$ can be seen by subtracting the expression for two different cutoffs then undoing the integrations by parts.
\end{proof}

We used the following lemma above.
\begin{lemma}\label{L:SimpleLemma}
	Let $V \in H^1(\R^3)$. Then
	\begin{align*} 
		\begin{split}
		    \int_{\R^3} (u \cdot \grad u) \cdot V
			&=  - \int_{\R^3} (u \cdot \grad V) \cdot u.
		\end{split}
	\end{align*}
\end{lemma}
\begin{proof}
	Using the vector identity, $(u \cdot \grad u) \cdot V
	    = u \cdot \grad (V \cdot u) - (u \cdot \grad V)
	    \cdot u$ gives
	\begin{align*}
		\begin{split}
		    \int_{\R^3} (u \cdot \grad u) \cdot V
			&= \int_{\R^3} u \cdot \grad (V \cdot u)
			    - \int_{\R^3} (u \cdot \grad V) \cdot u
			=  - \int_{\R^3} (u \cdot \grad V) \cdot u,
		\end{split}
	\end{align*}
	where the one integral vanishes since $\dv u = 0$.
\end{proof}

\section*{Acknowledgements}
\noindent DMA gratefully acknowledges support from the National Science Foundation through grant DMS-1907684.  EC 
gratefully acknowledges support by the Simons Foundation through Grant No. 429578.

\end{document}